\documentclass{amsart}

\usepackage{enumerate}
\usepackage{relsize} 
\usepackage{hyperref}
\usepackage[bbgreekl]{mathbbol}
\usepackage{verbatim}
\usepackage{geometry}
\usepackage{amsfonts}
\usepackage{amssymb}
\usepackage{amsmath}
\usepackage{amsthm}
\newcommand{\triplearrows}{\begin{smallmatrix} \to \\ \to \\ 
\to \end{smallmatrix} }

\newcommand{\cl}{\colon}
\newcommand{\itors}{I-\mathrm{tors}}
\newcommand{\mdc}[1]{\md(#1)_{>-\infty}^{\mathrm{cpl}}}
\newcommand{\mdn}[1]{\md(#1)^{\mathrm{nil}}_{>-\infty}}

\newcommand{\tfft}{\mathrm{Alg}^{\mathrm{tfp}, \flat}}
\newcommand{\aperf}{\mathrm{APerf}}
\newcommand{\arc}{\mathrm{arc}}
\newcommand{\qcoh}{\mathrm{QCoh}}
\newcommand{\md}{\mathrm{Mod}}
\usepackage{cleveref}

\usepackage{color}
\definecolor{todo}{rgb}{1,0,0}
\definecolor{conditional}{rgb}{0,1,0}
\definecolor{e-mail}{rgb}{0,.40,.80}
\definecolor{reference}{rgb}{.20,.60,.22}
\definecolor{mrnumber}{rgb}{.80,.40,0}
\definecolor{citation}{rgb}{0,.40,.80}

\DeclareMathOperator{\spec}{Spec}

\DeclareMathOperator{\spf}{Spf}

\renewcommand{\hom}{\mathrm{Hom}}
\newcommand{\vect}{\mathrm{Vect}}

\newcommand{\fun}{\mathrm{Fun}}

\newcommand{\perf}{\mathrm{Perf}}
\usepackage{xy}
\input xy

\xyoption{all}

\newcommand{\sF}{\mathcal{F}}

\theoremstyle{definition}
\newtheorem{definition}{Definition}[section]

\newtheorem{construction}[definition]{Construction}
\newtheorem{example}[definition]{Example}
\newtheorem{remark}[definition]{Remark}
\theoremstyle{theorem}
\newtheorem{proposition}[definition]{Proposition}
\newtheorem{lemma}[definition]{Lemma}
\newtheorem{corollary}[definition]{Corollary}

\newtheorem{theorem}[definition]{Theorem}

\renewcommand{\L}{}

\renewcommand{\phi}{\varphi}

\newtheoremstyle{named}{}{}{\itshape}{}{\bfseries}{.}{.5em}{#1 \thmnote{#3}}
\theoremstyle{named}

\newcommand{\rngp}{\mathrm{Alg}_{\mathcal{O}_K}^{\flat}}
\begin{document}

\title{Faithfully flat descent of almost perfect complexes in rigid geometry}
\author{Akhil Mathew}
\maketitle

\begin{abstract}
We prove a version of faithfully flat descent in rigid analytic geometry, for almost perfect complexes
and without finiteness assumptions on the rings involved. 
This extends results of Drinfeld for vector bundles.

\end{abstract}

\newcommand{\coh}{\mathrm{Coh}}

\section{Introduction}

We begin with the statement of faithfully flat descent in rigid analytic
geometry. 
Let $K$ be a complete nonarchimedean field, and let 
$A \to A'$  be a faithfully flat map of $K$-affinoid algebras. 
In ordinary algebra, faithfully flat descent \cite[Exp. VIII]{SGA1} states that the category of $A$-modules can be described
as the category of $A'$-modules with \emph{descent data}, which take place over the
tensor products $A'
\otimes_A A', A' \otimes_A A' \otimes_A A'$.  
In rigid geometry, one obtains a similar conclusion, but with the tensor
products replaced with completed tensor
products and only for finitely generated modules. 

Given a $K$-affinoid algebra $B$, we recall that $B$ is noetherian, and we let $\coh(B)$
denote the category of finitely generated $B$-modules. 
One has the following result, due to Bosch-G\"ortz \cite{BG};
the discretely valued case was previously known to Gabber (cf.~\cite[Theorem
1.9]{Ogus}). See also \cite{Conrad} and
\cite[Sec.~5.11]{EGR} for accounts.

\begin{theorem}[Bosch--G\"ortz--Gabber] 
\label{fflatrigid}
Let $A \to A'$ be a faithfully flat map of $K$-affinoid algebras. 
We have an equivalence of categories
\[  \coh(A) \simeq \varprojlim \left( \coh(A') \rightrightarrows \coh(A'
\hat{\otimes}_A A') \triplearrows \dots \right).   \]
In other words, to specify a  finitely generated $A$-module is equivalent to specifying a
finitely  generated $A'$-module with descent data over the completed tensor products $A'
\hat{\otimes}_A A', A' \hat{\otimes}_A A' \hat{\otimes}_A A'$. 
\end{theorem} 

There are at least two ways in which one could hope to generalize 
\Cref{fflatrigid}. 
The first is to work with the derived $\infty$-category; in the algebraic
setting, 
one can generalize faithfully flat descent to derived $\infty$-categories
\cite[Sec.~D.6.3]{SAG}. 
In the rigid analytic case, one imposes the following finiteness condition,
introduced in \cite[Exp.~I]{SGA6}. 

\begin{definition}[Almost perfect complexes] 
For a ring $R$, we 
consider the derived $\infty$-category $D(R)$. An object $M \in D(R)$ is called 
\emph{almost perfect} (or \emph{pseudocoherent}) 
if it can be represented (up to quasi-isomorphism) by  a chain complex
$M_\bullet$ such that $M_i = 0$ for $i \ll 0$ and each $M_i$ is finitely
generated projective. We let $\aperf(R) \subset D(R)$ be the full subcategory
spanned by almost perfect objects. 
\end{definition} 

\begin{example} 
\label{noethaperf}
Suppose $R$ is noetherian. 
Then $M \in D(R)$ is almost perfect if and only if $H_i(M )$ is finitely
generated for each $i$ and vanishes for $i \ll 0$. 
In particular, $\aperf(R)$ is the (homologically) bounded-below derived $\infty$-category of the
abelian category of finitely generated $R$-modules. 
\end{example}

From \Cref{fflatrigid} and in view of 
\Cref{noethaperf}, one may deduce the following extension of
faithfully flat descent to almost perfect complexes; 
compare Hennion--Porta--Vezzosi \cite[Sec.~3]{HPV} for closely related results. 
\begin{theorem} 
\label{aperfrigiddesc}
The construction $A \mapsto \aperf(A)$ satisfies flat hyperdescent on $K$-affinoid
algebras.\footnote{One also has descent for hypercovers instead only for
\v{C}ech covers; this is additional information when one works with
$\infty$-categories.} In particular, if $A \to A'$ is a faithfully flat map of $K$-affinoid
algebras, then 
\[ \aperf(A) \xrightarrow{\sim} \varprojlim \left( \aperf(A') \rightrightarrows \aperf(A'
\hat{\otimes}_K A') \triplearrows \dots \right).  \]
\end{theorem}

The second way one may attempt to generalize \Cref{fflatrigid} is to allow more
general rings than $K$-affinoid algebras. 
The context of rigid geometry imposes strong finiteness assumptions: in particular, the rings $A, A'$ are
noetherian.  
A result of Drinfeld shows that descent for \emph{vector bundles}
(rather than finitely generated modules) holds very generally. 
We next formulate a version of Drinfeld's theorem. 

\begin{definition}[The site $\rngp$]
\label{rngpsite}
Let $\mathcal{O}_K \subset K$ be the ring of integers, and let $\pi \in
\mathcal{O}_K$ denote a nonzero nonunit. 
Let $\rngp$ denote the category of $\mathcal{O}_K$-algebras $R$ 
which are $\pi$-torsion-free and $\pi$-adically complete. 
We say that a map
$R \to R'$ in $\rngp$ is \emph{$\pi$-completely faithfully flat} if $R/\pi \to
R'/\pi$ is faithfully flat; this defines the $\pi$-completely flat topology on
$(\rngp)^{op}$. 
\end{definition}

For any ring $A$, we let $\vect(A)$ 
denote the category of finitely generated projective $A$-modules. 
Then one has the following result. Compare 
\cite[Th. 3.11]{Dri06} and \cite[Prop. 3.5.4]{Dri18}. 
\begin{theorem}[{Drinfeld}] 
\label{Drinfeldthm}
The construction $R \mapsto \vect( R[1/\pi])$ is a sheaf 
of categories on $\rngp$. 
That is, given $R \to R'$ in $\rngp$ which is $\pi$-completely faithfully flat, the natural functor 
$$ \vect( R[1/\pi]) \to \varprojlim ( \vect( R'[1/\pi]) \rightrightarrows
\vect( \widehat{R' \otimes_R R'} [1/\pi]) \triplearrows \dots )$$ 
is an equivalence of categories. 
\end{theorem}

In this note, we will prove various common generalizations of
\Cref{Drinfeldthm} and \Cref{aperfrigiddesc}, using some simplifications that
occur when one works in the derived context. 
An instance of the result is the following.  

\begin{theorem} 
\label{aperfringp}
The construction $R \mapsto \aperf( R[1/\pi])$ is a hypercomplete sheaf of
$\infty$-categories on $\rngp$. 
In particular, given $R \to R'$ in $\rngp$ which is $\pi$-completely faithfully flat, the natural functor 
 induces an equivalence
$$ \aperf( R[1/\pi]) \simeq \varprojlim ( \aperf( R'[1/\pi]) \rightrightarrows
\aperf( \widehat{R' \otimes_R R'} [1/\pi]) \triplearrows \dots ).$$ 
This remains true if we replace $\aperf$ with $\aperf_{\geq 0} \subset \aperf$ (the
subcategory of connective objects) or $\perf \subset \aperf$ (the subcategory
of perfect complexes), or the subcategories
$\perf_{[a,b]} \subset \perf$ of perfect complexes with $\mathrm{Tor}$-amplitude in $[a,b ]$. 
\end{theorem}

The passage from \Cref{fflatrigid} to \Cref{aperfrigiddesc} is
facilitated by the $t$-structure on $\aperf$ in the $K$-affinoid case.  
By contrast, in the setting of \Cref{aperfringp}, there is no $t$-structure on
$\aperf(R[1/\pi])$. As a consequence, we do not know how to deduce \Cref{aperfringp}
from \Cref{Drinfeldthm}, and we will give a different argument: we will
construct a stable $\infty$-category $\mathcal{M}(R)$ into which
$\aperf(R[1/\pi])$ embeds, and such that $\mathcal{M}(R)$ admits a
$t$-structure that enables one to prove descent results. 
Our 
main ingredients are the monadicity theorem of Barr--Beck--Lurie and a
tool that exists only in the derived sense: the
equivalence between torsion and complete modules \cite[Th.~2.1]{DG02}. 
Although it will not strictly be necessary to the proof, our argument is
inspired by a result of Bhatt \cite{Bha19} that in this case states that any
$\pi$-complete $\mathcal{O}_K$-module $M$ such that $M[1/\pi] = 0$ is in fact
annihilated by a fixed power of $\pi$. 

In fact, we will prove two 
versions of our descent result (\Cref{fflataperf} and \Cref{rigiddescuniv}
below). The first is a generalization 
of \Cref{aperfringp} to the case of a finitely generated ideal (rather than simply
inverting an element). The second is one where $\pi$-complete faithful flatness is
replaced by universal descent in the sense of \cite[Sec.~3]{MGal}, which for
finitely presented maps is equivalent to being a $v$-cover \cite{Rydh}.

\renewcommand{\perf}{\mathrm{Perf}}

Part of the difficulty in proving descent results such as the above is the lack 
of a well-behaved theory of quasicoherent (rather than coherent) sheaves in analytic
(e.g., rigid) geometry. Recently Clausen and Scholze have, using condensed
mathematics, defined a category of quasicoherent sheaves on an extremely large
class of analytic spaces \cite{AnalyticGeometry}.  We expect that our results can be obtained using
their framework. Compare also recent work
of Andreychev \cite{And21}, which uses  
condensed mathematics to prove
some related  descent statements (for the analytic rather than flat topology).

\subsection*{Conventions}
We will generally use the notations and conventions of \cite{HA, SAG}, and
formulate our results for connective $E_\infty$-rings rather than for
ordinary commutative rings.
In many cases this is essential even for questions that begin with discrete rings, because the derived completion process may
introduce higher homotopy groups. 
In particular, all tensor products will be derived tensor products. 

Given a connective $E_\infty$-ring $R$, 
we write $\md(R)$ for the $\infty$-category of $R$-module spectra. 
When $R$ is an ordinary commutative ring, this recovers the unbounded derived
$\infty$-category $D(R)$. 
The stable $\infty$-category $\md(R)$ is equipped with a $t$-structure, and we
let 
$\md(R)_{\geq 0}, \md(R)_{\leq 0} \subset \md(R)$ denote the subcategories of
connective and coconnective objects. 
We let $\md(R)^{\heartsuit}$ denote the (usual) abelian category of discrete
$\pi_0(R)$-modules, which is the heart of this $t$-structure. 
We will use \emph{homological} indexing conventions. 

We will freely use the theory of $t$-structures on stable
$\infty$-categories
as in \cite[Sec.~1.2.1]{HA} (after \cite{BBD} for triangulated categories). 
Given a stable $\infty$-category $\mathcal{C}$ equipped with a $t$-structure, we
say that $\mathcal{C}$ is \emph{right-bounded} if $\mathcal{C} = \bigcup_{n}
\mathcal{C}_{\geq n}$ and \emph{left-complete} if $\mathcal{C} \simeq
\varprojlim_n \mathcal{C}_{\leq n}$ (where the transition maps are the
truncation maps).

\subsection*{Acknowledgments} I would like to thank Benjamin Antieau, Bhargav
Bhatt, Dustin Clausen, 
Brian Conrad, Adriano  C{\'o}rdova,  Vladimir Drinfeld, Ofer Gabber, Aron Heleodoro, 
Arthur-C\'esar Le Bras, Jacob Lurie, Peter Scholze, and the referee for helpful discussions and comments, and the Institute for Advanced
Study for hospitality. This work was done while the
author was a Clay Research Fellow.

\section{Isogenies}

\subsection{Generalities}
Throughout, 
we fix a connective $E_\infty$-ring $R$ (e.g., a discrete ring $R$) and an ideal $I \subset \pi_0(R)$. 
Let $\mathcal{C}$ be an $R$-linear additive $\infty$-category. 
Most often, we will take $\mathcal{C}$ to be a subcategory of the
$\infty$-category $\md(R)$ of $R$-modules. 
Our basic ``meta-definition'' is the following. 
\begin{definition}[Properties up to isogeny]
\label{propisogeny}
Let $\mathcal{P} \subset \mathcal{C}$ be a full subcategory stable
under finite direct sums and retracts. 
We define a full subcategory $\mathcal{P}_{\leq I} \subset \mathcal{C}$ as follows: 
an object $X \in \mathcal{C}$ belongs to $\mathcal{P}_{\leq I}$ if 
for each $a \in I$, there exists $X_0 \in \mathcal{P}$ (depending on $a$)  and maps 
$f\cl  X \to X_0, g \cl  X_0 \to X$ with $g \circ f \cl  X \to X$ given by multiplication
by $a$. 
We also write $\mathcal{P}_{\leq I^\infty}= \bigcup_{n \geq 0} \mathcal{P}_{\leq
I^n}$; informally, we can think of $\mathcal{P}_{\leq I^\infty}$ as those
objects which satisfy the defining property of objects in $\mathcal{P}$ ``up to isogeny.''\end{definition} 

In this note, we will be interested almost exclusively in the case where
$I \subset \pi_0(R)$ is a \emph{finitely generated} ideal. By contrast, when $I = I^2$,  
this type of definition is frequently used in almost ring theory \cite{GR}.

\begin{remark} 
\begin{enumerate}
\item When $I = (1)$, then $\mathcal{P}_{\leq I} = \mathcal{P}$.  This follows because $\mathcal{P}$ is closed under retracts.

\item For convenience, we have formulated the above for $\infty$-categories, but
the above definition only depends on the underlying homotopy category
$\mathrm{Ho}(\mathcal{C})$ (which is naturally enriched in $\pi_0(R)$-modules). 
\item 
Let $F \cl  \mathcal{C} \to \mathcal{D}$ be an $R$-linear functor of
$R$-linear additive $\infty$-categories. 
Suppose $\mathcal{P} \subset \mathcal{C}, \mathcal{P}' \subset \mathcal{D}$ are
full subcategories closed under finite direct sums and
retracts. 
Suppose $F$ carries $\mathcal{P}$ into $\mathcal{P}'$. Then $F$ carries $\mathcal{P}_{\leq I}$ into
$\mathcal{P}'_{\leq I}$. 
\label{Iisogsorite}
\item Let $J \subset \pi_0(R)$ be another ideal. 
Then $(\mathcal{P}_{\leq I})_{\leq J} \subset \mathcal{P}_{\leq IJ}$. 
\item Let $\left\{I_\alpha, \alpha \in A\right\}$ by a filtered system of ideals
and let $I = \bigcup I_\alpha$. Then $\mathcal{P}_{ \leq I} = \bigcap
\mathcal{P}_{\leq I_\alpha}$. 
\end{enumerate}
\end{remark} 

It will be convenient to rephrase \Cref{propisogeny} in the finitely generated
case in terms of the notion of $\leq I$-split surjection (resp. injection), as
will follow from \Cref{newcritisog} below. 
\begin{definition}[$\leq I$-split surjections and injections] 
\begin{enumerate}
\item  
A map of (discrete) $\pi_0(R)$-modules $M \to N$ is said to be \emph{$\leq
I$-surjective} 
if its cokernel is annihilated by $I$. 

\item
A map $f \cl  X \to Y$ in $\mathcal{C}$ is said to be a \emph{$\leq I$-split
surjection} if for all $Z \in \mathcal{C}$, the map of $\pi_0(R)$-modules
$\pi_0\hom_{\mathcal{C}}( Z, X) \to \pi_0 \hom_{\mathcal{C}}(Z, Y)$ is 
$\leq I$-surjective. 
This is equivalent to the statement that  for each $a \in
I$, there exists $g_a \cl  Y \to X$ such that
$f \circ g_a \cl  Y \to Y$ is given by multiplication by $a$.  
\item 
A map $f \cl  X \to Y$ in $\mathcal{C}$ is said to be a \emph{$\leq I$-split
injection} if for all $Z \in \mathcal{C}$, the map of $\pi_0(R)$-modules
$\pi_0\hom_{\mathcal{C}}( Y, Z) \to \pi_0 \hom_{\mathcal{C}}(X, Z)$ is 
$\leq I$-surjective. 
This is equivalent to the statement that  for each $a \in
I$, there exists $g_a \cl  Y \to X$ such that
$ g_a\circ f \cl X \to X$ is given by multiplication by $a$.  

\end{enumerate}
\end{definition} 

\begin{remark} 
Let $J \subset \pi_0(R)$ be another finitely generated ideal. 
The composite of a $\leq I$-split surjective (resp. $\leq I$-split injective) map and
a $\leq J$-split surjective map (resp. $\leq J$-split injective map) is $\leq
IJ$-split surjective (resp. $\leq IJ$-split injective). 
\end{remark} 

\begin{proposition} 
\label{newcritisog}
Suppose $I \subset \pi_0(R)$ is a finitely generated ideal. 
Then the following are equivalent: 
\begin{enumerate}
\item  
An object $X \in \mathcal{C}$ belongs to $\mathcal{P}_{\leq I}$. 
\item There exists a $\leq I$-split surjection $X' \to X$ with $X' \in
\mathcal{P}$. 
\item
There exists  a $\leq I$-split injection $X \to X''$ with $X'' \in \mathcal{P}$. 
\end{enumerate}
\end{proposition} 
\begin{proof} 
It is easy to see from the definitions that (2) or (3) implies (1).  
Suppose (1). Let $x_1, \dots, x_n \in I$ be a set of generators. 
We have objects $X'_1, \dots, X'_n \in \mathcal{P}$  and maps $f_i \cl X \to X_i'$
and $g_i \cl X'_i \to X$ such that $g_i \circ f_i $ is multiplication by $x_i$. 
We can consider the map $\oplus_i f_i \cl X \to \bigoplus_{i = 1}^n X_i'$. 
Since the $x_i$ generate $I$, it is not difficult to see that this map is a
$\leq I$-split injection. Similarly, the map $\oplus_i g_i \cl \bigoplus_{i=1}^n
X_i' \to X$ is a $\leq I$-split surjection. 
\end{proof}

\begin{corollary} 
Suppose $I$ is finitely generated and $\mathcal{P}$ is the filtered union of the full additive,
idempotent-complete subcategories
$\mathcal{P}_\alpha, \alpha \in A$. Then $\mathcal{P}_{\leq I} = \bigcup_{\alpha
\in A} (\mathcal{P}_\alpha)_{\leq I}$.  
\end{corollary} 

We will use \Cref{propisogeny} in the following instances. 

\begin{definition} 
\begin{enumerate}
\item 
Suppose $\mathcal{P}$ is the 
subcategory of  zero objects. Then an object $X \in \mathcal{C}$ belongs to  $\mathcal{P}_{\leq I}$ if and only if 
every element of $I$ acts by zero on $X$. 
In this case, we will say that $X$ is \emph{$ \leq I$-isogenous to zero.}
\item 
Let $\mathcal{D} = \fun( \Delta^1, \mathcal{C})$ denote the $\infty$-category of arrows $X \to X'$ in $\mathcal{C}$, and
let 
$\mathcal{P}$ be the subcategory of isomorphisms. 
An arrow in $\mathcal{C}$ belonging to $\mathcal{P}_{\leq I}$ is said to be an
\emph{$\leq I$-isogeny}. \item 
Suppose $\mathcal{C}$ is a compactly generated $R$-linear additive $\infty$-category
 and
$\mathcal{P}$ is the subcategory of compact objects. We say that an object 
is \emph{$\leq I$-compact} if it belongs to $\mathcal{P}_{\leq I}$. 
\end{enumerate}
\end{definition} 

\newcommand{\mdd}{\mathrm{Mod}^{\heartsuit}}

For the next results, we note that there is a small subtlety that the construction of the arrow
$\infty$-category does not commute with taking homotopy categories, i.e.,
$\mathrm{Ho}( \fun(\Delta^1, \mathcal{C})) \neq \fun( \Delta^1,
\mathrm{Ho}(\mathcal{C}))$ (if so, the proofs could be shortened). 
Instead, we have
a fiber sequence
\begin{equation}  
\label{HominDelta1C}
\hom_{\fun( \Delta^1, \mathcal{C})}( X \to Y, X' \to Y')  \to 
\hom_{\mathcal{C}}(X, X') \times \hom_{\mathcal{C}}(Y, Y') \to
\hom_{\mathcal{C}}(X, Y'). 
\end{equation}

\begin{proposition}[Characterization of $\leq I$-isogenies] 
\label{leqIisogchar}
If an arrow $f \cl X \to Y$ is a $\leq I$-isogeny, then for   each $a \in I$, there
exists $g_a \cl Y \to X$ such that $f \circ g_a $ and $ g_a \circ f$ are given by
multiplication by $a$. 
Conversely, if this condition holds, then $f$ is a $\leq I^2$-isogeny. 
\end{proposition} 
\begin{proof} 
The first direction is a diagram-chase, which we leave to the reader. 
For the other direction, suppose
there exist  maps $g_a$ as in the statement. 
Consider 
the natural  morphism in $\fun(\Delta^1, \mathcal{C})$ given by $(X
\stackrel{{f}}{\to} Y) \to (Y \stackrel{\mathrm{id}}{\to} Y)$. 
One sees from 
\eqref{HominDelta1C}
that this map is a $\leq I^2$-split injection, which shows that $f$ is a $\leq
I^2$-isogeny. 
\end{proof} 

\begin{remark} 
Suppose $\mathcal{C}$ is actually a 1-category. Suppose $f \cl X \to Y$ is an arrow
in $\mathcal{C}$ such that there exist maps $g_a \cl Y \to X$ for $a \in I$ as in 
\Cref{leqIisogchar}. Then $f$ is a $\leq I$-isogeny; indeed, for each $a$, we consider the
diagram
\[ \xymatrix{
  X \ar[d]^f \ar[r]^a &  X \ar[d]^{\mathrm{id}} \ar[r]^{\mathrm{id}} & X
  \ar[d]^f    \\
  Y \ar[r]^{g_a} &  X \ar[r]^f &  Y
},\]
such that the composite map is multiplication by $a$ in the arrow category. 
In particular, if $\mathcal{C}$ is a 1-category, then \Cref{leqIisogchar} 
simplifies to an if and only if assertion. 
\end{remark}

We will need a variant of the five-lemma, whose proof we leave to the
reader. 
\begin{lemma} 
\label{fivelem}
Let 
\[ \xymatrix{
M_1 \ar[d]^{f_1}  \ar[r] &  M_2 \ar[d]^{f_2} \ar[r] &  M_3 \ar[d]^{f_3}
\ar[r] &  M_4 \ar[d]^{f_4}
\ar[r] &  M_5 \ar[d]^{f_5} \\
N_1 \ar[r] &  N_2 \ar[r] &  N_3 \ar[r] &  N_4 \ar[r] &  N_5
}\]
be a commutative diagram of discrete $\pi_0(R)$-modules with exact rows. 
Suppose $f_2, f_5$ are isomorphisms and $f_4$ is $\leq I$-surjective. Then $f_3$
is $\leq I$-surjective. 
\end{lemma} 

\begin{proposition} 
\label{arrowlem}
Let $f = (f_1, f_2) \cl (X \to Y) \to (X' \to Y')$ be a map in $\fun(\Delta^1,
\mathcal{C})$, leading to a commutative square
\[ 
\xymatrix{
X \ar[d]^{f_1}  \ar[r] &  Y \ar[d]^{f_2} \\
X' \ar[r] &  Y'
}
.\]
Suppose $f_1 \cl X \to X'$ is $\leq I$-split surjective and $f_2$ is an
equivalence. Then $f $ is $\leq I$-split surjective. 
Similarly, suppose $f_1$ is an equivalence and $f_2$ is $\leq I$-split
injective. Then $f$ is $\leq I$-split injective. 
\end{proposition} 
\begin{proof} 
Combine \Cref{fivelem} and  the fiber sequence \eqref{HominDelta1C}. 
\end{proof}

\begin{proposition} 
Let $\mathcal{P}, \mathcal{Q} \subset \mathcal{C}$ and suppose $I, J \subset
\pi_0(R)$ be ideals. 
Let $\fun(\Delta^1, \mathcal{C})$ be the $\infty$-category of arrows in
$\mathcal{C}$. 
Let $\mathcal{R} \subset \fun(\Delta^1, \mathcal{C})$ be the subcategory of
arrows $X \to Y$ with $X \in \mathcal{P}, Y \in \mathcal{Q}$. 
Then an arrow
$X' \to Y'$ in $\fun(\Delta^1, \mathcal{C})$ with $X' \in \mathcal{P}_{\leq I}$
and $Y' \in \mathcal{Q}_{\leq J}$ belongs to $\mathcal{R}_{\leq IJ}$. 
\label{arrowIJ}
\end{proposition} 
\begin{proof} 
Without loss of generality, we can assume $I, J$ are finitely generated. 
By assumption and \Cref{newcritisog}, there is a $\leq I$-split surjection $W' \to X'$ with $W' \in
\mathcal{P}$. 
We then get a map of arrows $(W' \to Y') \to (X' \to Y')$ which is necessarily
a $\leq I$-split surjection in 
$\fun(\Delta^1, \mathcal{C})$ by \Cref{arrowlem}. 
Therefore, it suffices to show that $(W' \to Y') \in \mathcal{R}_{\leq J}$. 
But (by \Cref{newcritisog} again) we have a $\leq J$-split injection $Y' \to Z'$ with $Z' \in \mathcal{P}$. 
We get a map $(W' \to Y') \to (W' \to Z')$ which is a $\leq J$-split injection
via 
\Cref{arrowlem}. However, $(W' \to Z') \in \mathcal{R}$, so we conclude. 
\end{proof} 

Our main application of 
\Cref{arrowIJ} is that the $\leq I$-construction behaves well with respect to
extensions. 

\begin{corollary} 
\label{Iextensions}
Suppose $\mathcal{C}$ is an $R$-linear stable $\infty$-category, and let $I , J
\subset \pi_0(R)$ be finitely generated ideals. 
Let $\mathcal{P}, \mathcal{Q} , \mathcal{R}$ be full additive,
idempotent-complete subcategories. 
Suppose the cofiber of any map with source in $\mathcal{P}$ 
and target in $\mathcal{Q}$ belongs to $\mathcal{R}$. 
Then the cofiber of any map with source in $\mathcal{P}_{\leq I}$ and target in
$\mathcal{Q}_{\leq J}$ belongs to $\mathcal{R}_{\leq IJ}$. 
\end{corollary} 
\begin{proof} 
This is a consequence of \Cref{arrowIJ}. 
We consider the cofiber functor $\mathrm{cofib} \cl \fun(\Delta^1, \mathcal{C}) \to
\mathcal{C}$. 
Let $\mathcal{W} \subset \fun(\Delta^1, \mathcal{C})$ be the subcategory of
arrows with source in $\mathcal{P}$ and target in $\mathcal{Q}$. 
The cofiber functor carries $\mathcal{W}$ into $\mathcal{R}$ and therefore
$\mathcal{W}_{\leq IJ}$ into $\mathcal{R}_{\leq IJ}$. By \Cref{arrowIJ}, any
arrow with source in $\mathcal{P}_{\leq I} $ and target in $\mathcal{Q}_{\leq
J}$ belongs to $\mathcal{W}_{\leq IJ}$. 
\end{proof} 

\begin{corollary} 
\label{leqIinfinitythick}
Suppose $\mathcal{C}$ is an $R$-linear stable $\infty$-category and $\mathcal{P}
\subset \mathcal{C}$ a thick  subcategory. 
Then $\mathcal{P}_{\leq I^\infty} \subset \mathcal{C}$ is also a thick
subcategory.  \qed
\end{corollary} 

\newcommand{\I}{\mathcal{I}}
\begin{proposition}[Characterization of $\leq I$-compact objects] Let
$\mathcal{C}$ be a compactly generated $R$-linear additive $\infty$-category. Then the
following are equivalent for $X \in \mathcal{C}$: 
\begin{enumerate}[a)]
\item $X$ is $\leq I$-compact.  
\item For every filtered system $Y_i, i \in \I$ in $\mathcal{C}$, the map 
\begin{equation}  
\label{cclim1}
\varinjlim \pi_0 \hom_{\mathcal{C}}( X, Y_i) \to  \pi_0 \hom_R (X,
\varinjlim Y_i) \end{equation} 
is an $\leq I$-isogeny of $\pi_0 R$-modules. 
\item 
For every filtered system $Y_i, i \in \mathcal{I}$ in $\mathcal{C}$, the map 
 \begin{equation} \label{cclim2} \varinjlim  \hom_{\mathcal{C}}( X, Y_i) \to
 \hom_{\mathcal{C}} (X,
 \varinjlim Y_i)  \end{equation}
is an $\leq I$-isogeny of $ R$-modules. 
\end{enumerate}
\end{proposition} 
\begin{proof} 
Suppose a). 
For every filtered system $Y_i, i \in I$, we define a functor on 
$\mathcal{C}$ with values in arrows in $R$-modules,  given by
$M' \mapsto  ( \varinjlim \hom_{\mathcal{C}}(X, Y_i) \to \hom_{\mathcal{C}}(X,
\varinjlim Y_i))
$.  
This carries compact objects in $\mathcal{C}$ to isomorphisms, so it carries
$\leq I$-compact objects to $\leq I$-isogenies. Thus, a) implies c), and c)
clearly implies b). 
To see that b) implies a), we write $X$ as a filtered colimit of compact objects
and use that \eqref{cclim1} is a $\leq I$-isogeny to exhibit $X$ as $\leq
I$-compact. 
\end{proof}

\subsection{Modules}
We now specialize to the case where $\mathcal{C}$ is given by modules over a
connective $E_\infty$-ring $R$ (or some appropriate subcategory), and begin by
reviewing some finiteness conditions. 
\begin{definition}[$\leq I$-finitely generated modules] 
Suppose $\mathcal{C} = \mdd(R)$  is the category of discrete $R$-modules, and $\mathcal{P}$ is the subcategory of finitely
generated modules.   Then we say that a discrete  $R$-module $M$ is \emph{$\leq
I$-finitely generated}
if it belongs to $\mathcal{P}_{\leq I}$. It is not difficult to see that this holds if
and only if $M$ is $\leq I$-isogenous to a finitely generated module. 
\end{definition}

\begin{definition}[Perfect and almost perfect modules] 
Let $\perf(R) \subset \md(R)$ denote the $\infty$-category of perfect
$R$-modules, or equivalently the compact objects in $\md(R)$.  
We let $\aperf(R) \subset \md(R)$ denote the subcategory of 
almost perfect (or pseudocoherent) $R$-modules (see \cite[Sec.~7.2.4]{HA}; in
the discrete case the definition is due to \cite[Exp.~I]{SGA6}). 
Given a qcqs spectral scheme $X$, we let $\qcoh(X)$ denote the $\infty$-category
of quasi-coherent modules on $X$, and $\perf(X), \aperf(X) \subset \qcoh(X)$ the
associated subcategories of perfect and almost perfect objects. 
\end{definition}

Our goal is to study $\leq I^\infty$-versions of perfectness and almost
perfectness. 
For this, we will need to use the following intermediate property. 
We review the notion of being ``perfect to order $n$,'' as in 
\cite[Sec.~2.7]{SAG} or \cite[Tag 064N]{stacks-project}, and then formulate its
$\leq I$-analog. 

\begin{definition} 
We say that a bounded-below object $M \in \md(R)$ is \emph{perfect to order $n$} 
if  the following equivalent (by \cite[Prop. 2.7.0.4]{SAG}) conditions are satisfied: 
\begin{enumerate}
\item For every filtered system $N_i, i \in \I$ in $\md(R)_{\leq n}$, the natural
map of $R$-module spectra
\begin{equation} \label{colimmap}  \varinjlim \hom_R( M, N_i) \to \hom_R( M, \varinjlim N_i)
\end{equation}
has homotopy fiber in $\md(R)_{<0}$: that is, it induces an isomorphism on
$\pi_i$ for $i > 0$ and an injection on $\pi_0$. 
\item
For every 
filtered system $N_i, i \in \I$ in $\md(R)_{\leq n }$ such that each transition
map $N_i \to N_j$ induces an injection on $\pi_n$, 
the natural map 
\eqref{colimmap} is an isomorphism on connective covers. 
\item For every filtered system $N_i, i \in \I$ of \emph{discrete}
$\pi_0(R)$-modules, the natural map 
\[ \varinjlim \mathrm{Ext}^i_R(M, N_i) \to \mathrm{Ext}^i_R(M, \varinjlim N_i)  \]
is an isomorphism for $i < n$ and an injection for $i = n$. 
\end{enumerate}
The condition of being perfect to order $n$ only depends on the truncation $\tau_{\leq n} M \in
\md(R)_{\leq n}$, so we will often view being ``perfect to order $n$'' as a property of objects in 
$\md(R)_{\leq n}$. 
Note finally that 
a bounded-below $R$-module $M$ is almost perfect if and only if it is perfect to
each order $n$. 
\end{definition}

\begin{example} 
An object $M \in \md(R)_{\geq 0}$ is perfect to order zero if and only if
$\pi_0(M)$ is a finitely generated $\pi_0(R)$-module. 
\end{example} 
\begin{example} 
\label{ex:perftoordern}
Let $M$ be a compact object of $\md(R)_{\leq n}$ (so $\tau_{\leq n}$ of a
perfect 
$R$-module). 
Choose a surjection of discrete $\pi_0(R)$-modules, $\pi_n M
\twoheadrightarrow M_n'$; then the pushout
$M \sqcup_{ (\pi_n M)[n] } M_n'[n]$ is an example of an object in
$\md(R)_{[0, n]}$ which is perfect
to order $n$. 
This is straightforward to see using the above criteria. 
Conversely, any object of $\md(R)_{[0, n]}$ which is perfect to order $n$ arises
in the above fashion. See \cite[Cor. 2.7.2.2]{SAG}.  
\end{example}

\begin{proposition} 
\label{extperfecttoorder}
Let $M' \to M \to M''$ be a cofiber sequence of bounded-below $R$-modules. 
Suppose $M$ is 
perfect to order $ n$. Then the following are equivalent: 
\begin{enumerate}
\item $M'$ is perfect to order $n-1$.  
\item $M''$ is perfect to order $n$. 
\end{enumerate}
\end{proposition} 
\begin{proof} 
This follows from a diagram-chase. 
Let $N_i, i \in I$ be a filtered system in $\md(R)_{[0, n]}$. 
Let $F$ be the functor which carries an $R$-module $Q$ to $\mathrm{fib}(
\varinjlim \hom_R(Q, N_i) \to \hom_R( Q, \varinjlim N_i))$. 
By assumption, $F(M) \in \md(R)_{<0}$; since we have a fiber sequence 
$F(M'') \to F(M) \to F(M')$, the result now follows easily. 
\end{proof}

\begin{definition}[Modules $\leq I$-perfect to order $n$] 
\label{Iperfectordern}
Let $M \in \md(R)_{\leq n}$. 
If $\mathcal{P} \subset \md(R)_{\leq n}$ is the subcategory of  objects which
are perfect to order $n$, 
then we will say that $M$ is \emph{$\leq I$-perfect to order $n$}  if $M$ belongs
to 
$\mathcal{P}_{\leq I}$. 
Given an arbitrary $R$-module $M$, we will say that 
$M$ is 
\emph{$\leq I$-perfect to order $n$}  if $\tau_{\leq n} M$ is. 
An object is \emph{$\leq I^\infty$-perfect to order $n$} if it is $\leq I^r$-perfect to
order $n$, for some $r$. 
\end{definition} 

\begin{proposition} 
Let $M \in \md(R)$ be $\leq I^\infty$-perfect to order $n$. 
Let $N \in \md(R)_{\leq n-1}$. 
Then for $t \in I$, the map 
$\tau_{\geq 0} \hom_{R}( M, N)[1/t] \to  \tau_{\geq 0} \hom_R( M[1/t], N[1/t])$ is an
isomorphism. 
\label{invertingtinI}
\end{proposition} 
\begin{proof} 
Note that this question depends only on $\tau_{\leq n-1} M$. 
Our assumption implies that $\tau_{\leq n-1} M$ receives a $\leq
I^\infty$-split surjection from $\tau_{\leq n-1}$ of a perfect $R$-module $P$.
Therefore, we find that $\tau_{\geq 0}\hom_R( M, N)[1/t]$ and $ \tau_{\geq 0} \hom_R( M[1/t], N[1/t])$ are 
naturally
retracts of 
$\tau_{\geq 0}\hom_R( P, N)[1/t]$  and $
\tau_{\geq 0}
\hom_R( P[1/t], N[1/t])$. 
Since
$\tau_{\geq 0} \hom_{R}( P, N)[1/t] \to  \tau_{\geq 0} \hom_R( P[1/t], N[1/t])$
is an isomorphism, we obtain the result for $M$ as well. 
\end{proof}

\begin{proposition} 
\label{critforIperfectn}
Let $M \in \md(R)_{\leq n}$ be bounded-below. 
In order for $M$ to be $\leq I$-perfect to order $n$, it is necessary and
sufficient that for every
filtered system $N_i, i \in \I$ in $\md(R)_{\leq n}$ with injective transition
maps on $\pi_n$, the connective cover of the natural map \eqref{colimmap} 
is an $\leq I$-isogeny. 
\end{proposition} 
\begin{proof} 
Without loss of generality, we can assume $I$ finitely generated. 
Necessity is clear, since the connective cover of \eqref{colimmap} is an
equivalence for $M$ perfect to order $n$ under our assumptions. 
We can find a filtered system $N_i \in \md(R)_{\leq n}$ with injective
transition maps on $\pi_n$ such that each $N_i$ is perfect to order $n$ and
such that $\varinjlim N_i \simeq M$, 
e.g., using \Cref{ex:perftoordern} and that $M$ is a filtered colimit of
perfect modules. Our assumption now shows that one of the
$N_i$ maps to $M$ via a $\leq I$-split surjection, which is enough to imply the
claim. 
\end{proof} 

\begin{example} 
\label{Iperftozero}
An object $M \in \md(R)$
 is $\leq I$-perfect to order zero if and only if $\pi_0(M)$ is $\leq
 I$-finitely generated. 
\end{example} 

\begin{proposition} 
\label{prop:shiftIperfect}
Let $M' \to M \to M''$ 
be a cofiber sequence in $\md(R)$ with all terms bounded-below. 
Suppose $M$ is perfect to order $n$. Then the following are equivalent: 
\begin{enumerate}
\item $M'$ is $\leq I$-perfect to order $ n-1$.  
\item  $M''$ is $\leq I$-perfect to order $n$. 
\end{enumerate}
\end{proposition} 
\begin{proof} 
This follows from 
\Cref{extperfecttoorder}, in light of \Cref{Iextensions}. 
Alternatively, one can argue directly using \Cref{critforIperfectn}. 
\end{proof}

\subsection{Passage from the generic fiber}

In this subsection, we will prove various instances of the following principle:
under (derived) $I$-completeness assumptions, 
for a module
to have a certain property outside $I$ is equivalent to having it up to
$\leq I^\infty$-isogeny  integrally, or equivalently up to $\leq I^r$ modulo $I^n$ for every
$n$ (where $r$ is independent of $n$).

To begin with, we consider the case of discrete modules over a commutative ring. 
Fix a commutative ring $R$ and a finitely generated ideal $I \subset R$. 
Recall the notion of \emph{derived completeness};
see \cite[Tag 091N]{stacks-project} or \cite[Sec.~7.3]{SAG} for accounts. 
\begin{definition}[Derived $I$-complete modules] 
A discrete  $R$-module $M \in \md(R)^{\heartsuit}$ is \emph{derived $I$-complete} if 
for each $x \in I$, we have 
$\varprojlim^i( \dots \xrightarrow{x} M \xrightarrow{x} M ) = 0$
for $i = 0, 1$. 
The collection of derived $I$-complete modules forms an abelian subcategory of
$\md(R)^{\heartsuit}$ which is closed under extensions, and which
properly contains    all  
classically $I$-complete modules. 
The inclusion of derived $I$-complete modules into all of $\md(R)^{\heartsuit}$
admits a left adjoint given by derived $I$-completion (and truncation in degree
zero), $M \mapsto \tau_{\leq 0}( \widehat{M}_I)$. 

\end{definition}

A derived $I$-complete module $M$ is generally not $I$-adically separated (i.e.,
one may have $\bigcap I^n M \neq 0$); this is the key difference between derived
$I$-complete and classically $I$-complete modules. 
Nonetheless, given a derived complete
module $M \in \md(R)^{\heartsuit}$, if $M/IM = 0$ then $M = 0$. 
In the future, we will need a slight refinement of this fact, as follows. 
\begin{lemma} 
\label{compositederivedtcomplete}
Let $R$ be a commutative ring with a fixed element $t \in R$. 
Let $M, N, P$ be discrete $R$-modules and suppose $P$ is derived $t$-complete. 
Let $f: M \to N$ and $g: N \to P$ be maps. 
Suppose that both $f$ and $g$ are $t$-divisible elements 
of $\hom_{R}(M, N), \hom_{R}(N, P)$. 
Then $g \circ f = 0$. 
\end{lemma} 
\begin{proof} 
We will use the following basic fact: among $t$-torsion-free objects, derived
and classical completion coincide; this follows, e.g., from the formula below
(\Cref{towercompletion}). 
Clearly $g$ annihilates all $t$-power torsion elements of $N$, so it factors
through a map 
$\widehat{(N/N_{\mathrm{tors}})} \to P$. 
It thus suffices to show that the composite map 
$M \stackrel{f}{\to} N \to \widehat{(N/N_{\mathrm{tors}})}$ is zero. 
Indeed, this composite map is divisible by arbitrary powers of $t$, but 
$\widehat{(N/N_{\mathrm{tors}})}$ (as the derived, and hence classical, $t$-completion of a
$t$-torsion-free module) is $t$-adically separated, forcing the map to vanish. 
\end{proof} 

Our starting point for this section is the following ``uniform boundedness'' result. 
See also \cite[Lemma 2.2]{BB19} for a special case (when $t = p$). 

\begin{theorem}[{Bhatt \cite{Bha19}}] 
\label{bhatthm}
Let $R$ be a commutative ring containing a finitely generated ideal $I \subset
R$. 
Let $M$ be a derived $I$-complete discrete $R$-module. Then the following are
equivalent: 
\begin{enumerate}
\item  
$M$ vanishes away from
$I$: that is, $M[1/t] = 0$ for $t \in
I$. 
\item
There exists $r$ such that $M$ is $\leq I^r$-isogenous to zero. 
\item There exists $r$ such that for all $n$, $M/I^n M$ is $\leq I^r$-isogenous 
to zero. 
\end{enumerate}
\end{theorem} 
\begin{proof} 
The equivalence of (1) and (2) appears in \cite{Bha19}. 
It remains only to 
show that (3) implies (2). 
Taking $n  = r+1$, we find that $I^r M = I^{r+1} M  = I( I^r M)$. 
Since $I^r M$ is derived $I$-adically complete (as the image of a map from
a finite direct
sum of copies of $M$ to $M$), it follows that $I^r M = 0$, as desired. 
\end{proof} 

\begin{proposition} 
\label{cor:fgaway}
Let $R$ be a ring which is derived $I$-complete for some finitely generated
ideal $I \subset R$. Let $M$ be a derived $I$-complete discrete $R$-module. Then the
following are equivalent: 
\begin{enumerate}
\item $M$ is finitely generated outside $I$ (that is, for $t \in I$, $M[1/t]$ is
a finitely generated $R[1/t]$-module). 

\item  There exists $r$ such that $M$ is $\leq I^r$-finitely generated. 
\item 
There exists $r$ such that for all $n$, $M/I^n M$ is $\leq I^r$-finitely generated. 
\end{enumerate}
In fact, the best possible $r$ for (2) and (3) are the same. 
\end{proposition} 
\begin{proof} 
Suppose (1). Then there exists a finitely generated free $R$-module $F$ and a map $f 
\cl 
F
\to M$ which induces a surjection after inverting any $t \in I$ (it suffices to
check this as $t$ runs over a system of generators). 
By \Cref{bhatthm}, we have $I^r \mathrm{coker}(f) = 0$ for some $r$. 
Replacing $F$ with $\mathrm{im}(f)$, we find that $M$ is $\leq I^r$-finitely
generated, proving (2). 
Clearly (2) implies (1), so it suffices to show that (3) implies (2). 
Taking $n = r+1$, we have that $M/I^{r+1} M$ is $\leq I^r$-finitely generated. 
Therefore, there exists a finitely generated free $R$-module $F'$ and a map $g \cl F'
\to M$ such that $\mathrm{coker}(g)/I^{r+1} \mathrm{coker}(g)$ is annihilated by
$I^r$, and therefore $I^r \mathrm{coker}(g) = 0$ by derived completeness. 
Thus $\mathrm{im}(g) \subset M$ is a $\leq I^r$-isogeny and the result follows. 
\end{proof}

Now we switch to the non-discrete setting: let $R$ be a connective
$E_\infty$-ring, and let $I \subset \pi_0(R)$ be a finitely generated ideal. 
We briefly recall the theory of $I$-complete  and $I$-torsion objects in $\md(R)$,
from \cite{DG02}; see also \cite[Sec.~7.3]{SAG} for a detailed account. 
\begin{definition}[Complete and torsion modules] 
\begin{enumerate}
\item  
Given $M \in \md(R)$, $M$ is said to be \emph{$I$-torsion} if all the homotopy
groups are $I$-power torsion (i.e., every element in $\pi_*(M)$ is annihilated
by a power of $I$).
\item
 The $R$-module $M$ is said to be \emph{$I$-complete} 
if for each $x \in I$, $\varprojlim ( \dots \xrightarrow{x} M \xrightarrow{x} M
) = 0 $ in $\md(R)$. 
This holds if and only if each homotopy group $\pi_i(M)$ is derived
$I$-complete as a discrete $\pi_0(R)$-module. 
\item The inclusion of $I$-complete modules in $\md(R)$ admits a left adjoint, called
\emph{$I$-completion}. \end{enumerate}
 \end{definition} 

\begin{construction}[A formula for the completion] 
Suppose $I = (x_1, \dots, x_p)$. 
For each $n>0$ and $1 \leq i \leq p$, let $R//x_i = \mathrm{cofib}(x_i^n \cl R \to
R)$ and $R_n = R//x_1 \otimes_R \dots \otimes_R R//x_p$. 
The family $\left\{R_n\right\}_{n \geq 1}$ forms a tower of perfect $R$-modules. 
Given any $M \in \md(R)$, the natural map 
$M \to \varprojlim_n (M \otimes_R R_n)$ exhibits the target as the
$I$-completion of $M$. 

Note that this tower $\left\{R_n\right\}$ is not canonical (it relies on a
choice of generators of $I$) and it does not generally form a tower of
$E_\infty$-$R$-algebras, but only of $R$-modules. 
In the future, we will fix an explicit choice of generators of $I$ and thus of a
tower $\{R_n\}$. 
\label{towercompletion}
\end{construction} 

Throughout, we will need to use the following fundamental equivalence between
complete and torsion modules, which goes back 
to \cite[Th.~2.1]{DG02} (see also \cite[Th.~3.3.5]{HPS});  
cf.~\cite[Prop.~7.3.1.7]{SAG} for an account in the present setting. 
\begin{theorem} 
\label{completenadtorsionareequivalent}
Let $R$ be a connective $E_\infty$-ring, and let $I \subset R$ be a finitely
generated ideal as above. 
The functor of $I$-completion also induces an
equivalence between the subcategories of $I$-torsion modules and $I$-complete
modules inside $\mathrm{Mod}(R)$. 
\end{theorem}

Next, we prove the analogs of 
\Cref{cor:fgaway} in the context of module spectra. 
The main results are similar: perfectness to any order on the
generic fiber can be checked integrally (or modulo any power of the ideal in an
appropriate sense) up to bounded isogeny.

\begin{proposition} 
\label{connectivityIbound}
Let $R$ be a connective $E_\infty$-ring and $I = (x_1, \dots, x_p) \subset \pi_0(R)$ a finitely
generated ideal; fix a tower $\left\{R_n\right\}$ as in \Cref{towercompletion}. 
Let $M \in \md(R)$ be $I$-complete.  
Then the following are equivalent: 
\begin{enumerate}
\item  
For each $m < 0$, $\pi_m( M)$ is $\leq I^\infty$-isogenous to zero. 
\item
For each $m < 0$, there exists $r$ such that $\pi_m( M \otimes_R R_n)$ is
$\leq I^r$-isogenous to zero for all $n >0$. 
\item 
For each $m < 0$ and $t \in I$, $\pi_m( M[1/t]) = 0$. 
\end{enumerate}
\end{proposition} 
\begin{proof} 
It is easy to see that (1) implies (2) and (3). 
Furthermore, (3) implies (1) thanks to 
\Cref{bhatthm}. It remains to show that (2) implies (1), so suppose (2). 
Since $M \simeq \varprojlim_n (M \otimes_R R_n)$, we conclude via the
Milnor exact sequence that
each $\pi_m(M), m < -1$ is $\leq I^\infty$-isogenous to zero. 
It suffices to verify that 
$\pi_{-1}(M)$ is $\leq I^\infty$-isogenous to zero. 
Given what we have already shown,
the hypotheses of the theorem are invariant under replacing $M$ with $\tau_{\geq
-1} M$, so we may assume $M$ is $(-1)$-connective. 
Then we find that $\pi_{-1}( M \otimes_R R_n) = \pi_{-1}(M)/(x_1^n, \dots, x_p^n) \pi_{-1}(M)$ is annihilated
by a fixed power of $I$ uniformly in $n$, 
 which forces (by \Cref{bhatthm}) $\pi_{-1}(M)$ to be $\leq I^\infty$-isogenous to zero. 
\end{proof}

\begin{proposition} 
\label{perfecttoordermI}
Let $R$ be a connective $E_\infty$-ring which is $I$-complete for $I = (x_1,
\dots, x_p)
\subset \pi_0(R)$ a finitely generated ideal; fix $\left\{R_n\right\}$ as in
\Cref{towercompletion}. 
Let $M$ be an object of 
$\md(R)$ which is bounded-below and $I$-complete. 
Given $m \geq 0$, the following are equivalent: \begin{enumerate}
\item For each $t \in I$, $M[1/t] \in \md(R[1/t])$ is perfect to order
$m$. 
\item $M$ is $\leq I^\infty$-perfect to order $m$. 
\item There exists $r'$ such that for all $n$, $M \otimes_R R_n \in \md(R)$ is $\leq
I^{r'}$-perfect to order $m$. 
\end{enumerate}
\end{proposition} 
\begin{proof} 
Without loss of generality (i.e., by shifting), we may assume $M$ connective. 
We use induction on $m$. 
In the case $ m = 0$, all three conditions are equivalent to 
$\pi_0(M)$ being $\leq I^\infty$-finitely generated 
(via \Cref{cor:fgaway}). 
Now suppose $m > 0$. 
Since all three conditions are isogeny invariant, we 
can assume without loss of generality that $\pi_0(M)$ is finitely generated. 
Choose a finitely generated free $R$-module $P$ and a map $P \to M$ inducing a
surjection $\pi_0(P) \twoheadrightarrow \pi_0(M)$, and write $F = \mathrm{fib}(P
\to M)$. 
Thanks to \Cref{prop:shiftIperfect}, conditions (1), (2), and (3) for $M$ are equivalent to  the
analogous conditions among (1), (2), and (3) for
$F$ (with $m$ replaced by $m-1$). By induction on $m$, all of these three
conditions are equivalent for $F$, and hence they are equivalent for $M$. 
\end{proof}

\section{The construction $\mathcal{M}(R)$}
Throughout this section, we fix a $E_\infty$-ring $R$ equipped with a finitely generated ideal $I \subset
\pi_0(R)$.  
We will construct a stable $\infty$-category $\mathcal{M}(R)$, which we
should regard as associated to the ``generic fiber'' of the formal
spectrum $\spf(R)$ (we will not need a precise notion here). 
Our main result is that $\aperf( \spec(R) \setminus V(I))$ is naturally a full subcategory of
$\mathcal{M}(R)$.

\subsection{Definition of $\mathcal{M}(R)$}
To begin with, we give the abelian version of the construction $\mathcal{M}(R)$;
in fact, there are two natural candidates, involving complete and torsion
modules respectively. In the derived setting, the distinction between the two
goes away.  
\begin{definition}[The construction $\mathcal{A}(R)$]
We let $\mathcal{A}(R)$ denote the abelian category obtained as the Serre
quotient of the category 
of $I$-power torsion discrete $\pi_0(R)$-modules by the subcategory of those modules
which are $\leq I^\infty$-isogenous to zero. 
Given a map $R \to R'$, 
base-change gives a right exact functor $\mathcal{A}(R) \to\mathcal{ A}(R')$
with a right adjoint (given by restriction of scalars) which is exact.
\end{definition} 

\begin{definition}[The construction $\mathcal{B}(R)$] 
We let $\mathcal{B}(R)$ denote the abelian category obtained as the Serre
quotient 
of the category 
of derived $I$-complete discrete $\pi_0(R)$-modules by the subcategory of those modules
which are $\leq I^\infty$-isogenous to zero. 
\end{definition} 

We can regard $\mathcal{B}(R)$ as an abelian version of 
the category of Banachian spaces studied in \cite{Dri18}. 
Next, we need the stable versions. 

\begin{construction}[$t$-structures on Verdier quotients] 
\label{tstructurequotient}
Let $\mathcal{E} $ be a  stable
$\infty$-category equipped with
a $t$-structure with heart
$\mathcal{E}^{\heartsuit}$. 
Given a Serre subcategory $\mathcal{B}_0 \subset  \mathcal{E}^{\heartsuit}$, we 
define a thick subcategory $\mathcal{E}' \subset \mathcal{E}$ consisting of
those objects all of whose homotopy groups lie in $\mathcal{B}_0$. 
Then $\mathcal{E}'$ also inherits a $t$-structure,
as does the Verdier quotient $\mathcal{E}/\mathcal{E}'$. The heart 
$\mathcal{E}'$ is given by $\mathcal{B}_0$ and the heart of 
of
$\mathcal{E}/\mathcal{E}'$ is given by the Serre quotient
$\mathcal{E}^{\heartsuit}/\mathcal{B}_0$. 
\end{construction} 
\begin{definition}[The constructions $\mathcal{M}_0(R), \mathcal{M}(R)$] 
\begin{enumerate}
\item  
Let $\mdc{R} \subset \md(R)$ denote the subcategory consisting of $R$-modules
which are bounded-below and $I$-complete.
Let $\mdn{R} \subset \mdc{R}$ denote the thick subcategory spanned by those $M$ such
that for each $i$, $\pi_i(M)$ is $\leq I^\infty$-isogenous to zero. 
A map with cofiber in $\mdn{R}$ is said to be a \emph{quasi-isogeny.}
\item
We have an equivalence of $\infty$-categories between $\mdc{R}$ and the $\infty$-category
of $R$-modules which are bounded-below and $I$-torsion
(\Cref{completenadtorsionareequivalent}, since both functors involved have
bounded amplitude). 
Therefore, $\mdc{R}$ is equipped with a
$t$-structure, which we refer to as the \emph{$I$-torsion $t$-structure}, whose
heart is the abelian category of $I$-power torsion discrete
$\pi_0(R)$-modules. This restricts to a $t$-structure on $\mdn{R}$ whose heart
is the abelian category of $\pi_0(R)$-modules which are $\leq
I^\infty$-isogenous to zero. 

\item 
We let $\mathcal{M}_0(R)$ denote the  Verdier quotient $\mdc{R}/\mdn{R}$, as an
$R$-linear stable $\infty$-category; it also
inherits a right-bounded (i.e., every object is bounded below) $t$-structure whose heart is given by the  
abelian category $\mathcal{A}(R)$ (\Cref{tstructurequotient}). 
It is linear over $\perf(R)$ and the action annihilates the $I$-torsion objects
in $\perf(R)$, so it becomes linear over $\perf( \spec(R) \setminus V(I))$. 

We let
$\mathcal{M}(R)$ denote its \emph{left completion} \cite[Sec.~1.2.1]{HA}.  
By construction, 
$\mathcal{M}(R) \simeq \varprojlim_n (\mathcal{M}_0(R))_{\leq n}$ where the
transition maps are the truncation functors. 
Similarly, $\mathcal{M}(R)$ is a stable $\infty$-category (linear over
$\perf(\spec(R) \setminus V(I))$) with a $t$-structure, which we refer to as the \emph{$I$-torsion
$t$-structure}; for each $n$, we have
$\mathcal{M}(R)_{\leq n} = \mathcal{M}_0(R)_{\leq n}$. 
\end{enumerate}
\end{definition} 

\begin{remark} 
There is no distinction between the idempotent-complete and the non-idempotent
complete Verdier quotient in the definition of $\mathcal{M}_0(R)$. Indeed, via \cite{Thomason}, this follows from the localization theorem in $K$-theory. 
It suffices to show that $K_0( \mdn{R}) =0$. 
This in turn follows  because for any object $X \in \mdn{R}$,
$\bigoplus_{i=1}^\infty X \in \mdn{R}$ as well, so that an Eilenberg
swindle implies $K_0(\mdn{R}) =0$. 
In particular, any truncated object in $\mathcal{M}(R)$ can be represented by an
object of $\mdc{R}$. 
\end{remark} 

\begin{remark} 
\label{replaceIcompl}
The constructions $\mdn{R}, \mdc{R}$ are insensitive to replacing $R$ by its
$I$-adic completion. 
\end{remark} 

\begin{construction}[The $I$-complete $t$-structure on $\mathcal{M}(R)$] 
The stable $\infty$-category $\mdc{R}$
also admits a $t$-structure obtained by restriction from the usual $t$-structure on
all $R$-modules, whose heart is given by derived $I$-complete discrete
$\pi_0(R)$-modules; we refer to this as the \emph{$I$-complete $t$-structure}. 
Similarly, using \Cref{tstructurequotient}, this
$t$-structure descends to $\mathcal{M}_0(R)$. We observe that it also descends
to $\mathcal{M}(R)$ with heart $\mathcal{B}(R)$. This follows because the $I$-torsion and $I$-complete
$t$-structures on $\mathcal{M}_0(R)$ (and on $\mdc{R}$) differ by a bounded amplitude. 
In particular, $\mathcal{M}(R)$ is also the left completion of
$\mathcal{M}_0(R)$ with respect to the $I$-complete $t$-structure. 
\end{construction}

\begin{construction}[Functoriality of $\mathcal{M}(R)$] 
The construction $R \mapsto \mathcal{M}(R)$ defines a functor 
in the connective $E_\infty$-ring $R$ (with choice of ideal $I \subset
\pi_0(R)$) to the $\infty$-category of right-bounded,
left-complete stable $\infty$-categories with $t$-structures (for \emph{either} choice
of $t$-structure), and right
$t$-exact functors between them; on hearts it is the base-change functor on
abelian categories
$\mathcal{A}(R), \mathcal{B}(R)$. 
Finally, given any map $R \to R'$ of connective $E_\infty$-rings, we get a right adjoint $\mathcal{M}(R') \to
\mathcal{M}(R)$ to the base-change functor (given by restriction of scalars) which is $t$-exact. 
\end{construction}

\subsection{Almost perfect complexes}

Throughout this subsection, we use the $I$-complete $t$-structure and we assume
$R$ is itself $I$-adically complete (no loss of generality by
\Cref{replaceIcompl}). 
It will be necessary to compute some mapping spaces in $\mathcal{M}(R)$; for
this, we use the following construction. 

\begin{construction}[Comparison of $\mathcal{M}(R)$ with the generic fiber] 
We have a $t$-exact functor (for the $I$-complete $t$-structure)
$\mdc{R} \to \qcoh( \spec(R) \setminus V(I))$ given by restricting to the locus
outside of $I$.
This clearly annihilates $\mdn{R}$, so we obtain a 
$t$-exact functor $\mathcal{M}_0(R) \to 
\qcoh( \spec(R) \setminus V(I))$. Since the target is left-complete, we 
obtain a factorization over a $t$-exact functor $j^* \cl \mathcal{M}(R) \to \qcoh( \spec(R)
\setminus V(I))$. 
\end{construction}

\begin{proposition} 
Suppose $R$ is $I$-complete. 
Suppose $X, Y \in \mdc{R}$. 
Suppose $X$ is  $\leq I^\infty$-perfect to order $n+1$ and $Y$ is $n$-truncated
in the $I$-complete $t$-structure on $\mdc{R}$. 
Then the functor $j^*$ induces an equivalence
of connective spectra
$\tau_{\geq 0}\hom_{\mathcal{M}_0(R)}(X, Y) \simeq 
\tau_{\geq 0}\hom_{\qcoh( \spec(R) \setminus V(I))}( j^* X,
j^* Y)$. 
\label{M0hom}
\end{proposition} 
\begin{proof} 
The $R$-modules 
$\hom_{\mathcal{M}_0(R)}(X, Y), 
\hom_{\qcoh( \spec(R) \setminus V(I))}( j^* X,
j^* Y)$
are both local away from $I$ (i.e., belong to the image of the fully faithful
embedding $\qcoh( \spec(R) \setminus V(I)) \to \md(R)$), since the $\infty$-categories
$\mathcal{M}_0(R), \qcoh( \spec(R) \setminus V(I))$ are naturally tensored over
$\perf( \spec(R) \setminus V(I))$. 
Thus, it suffices to show that 
$\tau_{\geq 0}\hom_{\mathcal{M}_0(R)}(X, Y) \to 
\tau_{\geq 0}\hom_{\qcoh( \spec(R) \setminus V(I))}( j^* X,
j^* Y)$ becomes an isomorphism after inverting any $t \in I$. 

Note that the $R$-module $Y$  is
$n$-truncated (with respect to the Postnikov $t$-structure). 
Therefore, the map 
$\tau_{\geq 0}\hom_{\md(R)}( X, Y)[1/t] \to \tau_{\geq 0}\hom_{\md(R)}( X[1/t], Y[1/t])$ is an isomorphism
by \Cref{invertingtinI}, since $X$ is $\leq I^\infty$-perfect to order $n+1$. 
To complete the proof, it suffices to show 
that for any $X' \to X$ in $\mdc{R}$ whose cofiber belongs to $\mdn{R}$, then
the map 
\[ \tau_{\geq 0}\hom_{\mdc{R}}(X', Y) \to \tau_{\geq 0}\hom_{\mdc{R}}(X, Y)  \]
induces an equivalence after inverting any $t \in I$. 
However, since $Y$  is $n$-truncated, we can identify this with the map 
\( \tau_{\geq 0}\hom_{\mdc{R}}(\tau_{\leq n}X', Y) \to
\tau_{\geq 0}\hom_{\mdc{R}}(\tau_{\leq n}X, Y)  \)
which becomes an isomorphism after inverting $t$, since $\tau_{\leq n} X' \to
\tau_{\leq n} X$ is a $\leq I^\infty$-isogeny. 
\end{proof}

\begin{corollary} 
\label{connectivtyg}
Suppose $R$ is $I$-complete. 
Suppose the ideal $I$ is generated by $g$ elements. 
Suppose $X \in \mathcal{M}(R)$ lives in degrees $\geq g$ with respect to the
$I$-complete $t$-structure. 
Then $\pi_0 \hom_{\mathcal{M}(R)}( R, X) = 0$. Here we identify $R$ with the
associated object of $\mathcal{M}(R)$. 
\end{corollary} 
\begin{proof} 
When $X$ is truncated, this follows from \Cref{M0hom}, since $\spec(R)
\setminus V(I)$ has cohomological dimension $\leq g-1$ and truncated objects
of $\mathcal{M}(R)$ and $\mathcal{M}_0(R)$ are identified.  
In general, we have that $\hom_{\mathcal{M}(R)}( R, X) = \varprojlim_n
\hom_{\mathcal{M}(R)}(R, \tau_{\leq n } X)$ from which the result follows via
the Milnor exact sequence. 
\end{proof}

Our goal is to identify a subcategory of ``weakly almost perfect objects'' in
$\mathcal{M}(R)$ and then to show that this subcategory is equivalent
(via $j^*$) to the $\infty$-category  $\aperf(
\spec(R) \setminus V(I))$. 
\begin{definition}[Objects weakly perfect to order $n$] 
We will say that an object of $\mathcal{M}_0(R)$ is \emph{weakly perfect to
order $n$} if 
for any representative  $X \in \mdc{R}$, $X$ is $\leq I^\infty$-perfect to order
$n$. 
This definition is clearly independent of the choice of representative $X$,
since any two are related by a zig-zag of quasi-isogenies. 
We will say that an object is \emph{weakly almost perfect} if it is weakly
perfect to any order. 

The condition that an object in $\mathcal{M}_0(R)$ should be weakly perfect to
order $n$ depends only on its $n$-truncation with respect to the $I$-complete
$t$-structure. 
Therefore, we also obtain analogous definitions of weakly almost perfect
(resp. weakly perfect to order $n$) for objects of $\mathcal{M}(R)$. 
\end{definition}

\begin{proposition} 
\label{surjweaklyperfect}
Suppose $R$ is $I$-complete. 
Let $X \in \mathcal{M}(R)_{\geq 0}$. 
\begin{enumerate}
\item  
$X$ is weakly perfect to order zero if and only if there exists a map
$R^r \to X$ (for some $r \geq 0$) whose cofiber $C$ belongs to $\mathcal{M}(R)_{\geq 1}$. 
\item
Suppose $n > 0$ and $X$ is weakly perfect 
to order $n-1$. Then $X$ is weakly perfect to order $n$ if and only if, for
any (or every)
map $R^r \to X $ with cofiber $C \in \mathcal{M}(R)_{\geq 1}$, we have
that $C[-1]$ is weakly perfect to order $n-1$. 
\end{enumerate}
\end{proposition} 
\begin{proof} 
The condition that $X$ should be weakly perfect to order $n$ only depends on the
$n$-truncation of $X$; moreover, maps $R^r \to X$ only depend on
$\tau_{\leq g} X$ (\Cref{connectivtyg}), if $I$ is generated by $g$ elements. 
Therefore, without loss of generality  
we may assume that $X \in \mathcal{M}_0(R)_{<\infty}$ and can be represented as
the image in $\mathcal{M}_0(R)$ of some $Y \in \mdc{R}$, which without loss of
generality we can take to be connective. Then $X$ is weakly
perfect to order zero if and only if there exists a map $R^r \to Y$ whose
cofiber in $\md(R)$
has $\pi_0$ which is $\leq I^\infty$-isogenous to zero, cf.~\Cref{Iperftozero}. 
 Assertion (1) now follows. Since any map of objects in $\mathcal{M}_0(R)$ can be represented by a
map in $\mdc{R}$ (subject to working up to quasiisogenies), assertion (2)
follows by 
\Cref{prop:shiftIperfect}.  
\end{proof}

\newcommand{\mdt}{\mathrm{Mod}_{\mathrm{tors}}(R)}
\begin{proposition} 
\label{expressasgeometricreal}
\begin{enumerate}
\item  
The collection of objects in $\mathcal{M}(R)_{\geq 0}$ which are weakly almost perfect
is closed under finite colimits and geometric realizations. 
\item 
An object of $\mathcal{M}(R)_{\geq 0}$ is weakly almost perfect if and only if
it can be written as a geometric realization of a simplicial object in
$\mathcal{M}(R)_{\geq 0}$ each of whose terms is a finite direct sum of copies
of $R$. 
\end{enumerate}
\end{proposition} 
\begin{proof} 
Part (1) follows from the fact that the condition of being $\leq
I^\infty$-perfect to some order $n$ is closed under finite colimits and depends
only on the $n$-truncation. 
Part (2) follows by iteration. 
Given a weakly almost perfect $X \in \mathcal{M}(R)_{\geq 0}$,
one produces (using \Cref{surjweaklyperfect}) a filtered object 
$0 = P_{-1} \to  P_0 \to P_1 \to \dots $ 
such that 
$P_i/P_{i-1}$ is equivalent to a finite direct sum of copies of $R[i]$
and whose colimit (which exists for connectivity reasons) is $X$. 
Converting this into a simplicial object via the Dold--Kan correspondence
\cite[Sec.~C.1.4]{SAG}, we see that (2) follows. 
\end{proof} 
\begin{proposition} 
Let $R$ be $I$-complete. 
\label{homsofweaklyap}
Suppose $X, Y \in \mathcal{M}(R)$  and $X$ is weakly almost perfect. 
Then $j^* X \in \qcoh( \spec(R) \setminus V(I))$ is almost perfect. Furthermore,
the natural map induces an equivalence
of $R$-module spectra
$$\hom_{\mathcal{M}(R)}(X, Y) \simeq \hom_{\qcoh( \spec(R) \setminus V(I))}( j^* X,
j^* Y).$$ 
\end{proposition} 
\begin{proof} 
Any object in $\mdc{R}$ which is $\leq I^\infty$-perfect to order $n$ restricts
to an object of $\qcoh( \spec(R) \setminus V(I))$ which is perfect to order $n$,
so the first assertion follows. 
Since $\mathcal{M}(R)$ is left-complete, 
it suffices to show for each $n$ that $$\tau_{\geq 0}\hom_{\mathcal{M}(R)}( \tau_{\leq n +1} X, \tau_{\leq n}
Y) \xrightarrow{\sim}
\tau_{\geq 0}\hom_{\qcoh( \spec(R) \setminus V(I))}( j^* \tau_{\leq n+1} X,
j^*\tau_{\leq n} Y) 
;$$ these (connective) mapping spectra can be computed in $\mathcal{M}_0(R)_{\leq n}$. 
Thus, the result follows from \Cref{M0hom} since we can represent $\tau_{\leq
n+1} X$ by an $I$-complete $R$-module which is perfect to order $n+1$. 
\end{proof}

\begin{proposition} 
\label{jgivesequiv}
The functor $j^*$ establishes a natural, symmetric monoidal equivalence between 
the weakly almost perfect objects of $\mathcal{M}(R)_{\geq 0} $ and $\aperf( \spec(R)
\setminus V(I))_{\geq 0}$. 
Similarly, $j^*$ establishes
a natural, symmetric monoidal equivalence between 
the weakly almost perfect objects of $\mathcal{M}(R) $ and $\aperf( \spec(R)
\setminus V(I))$. 

\end{proposition} 
\begin{proof} 
Full faithfulness is a consequence of \Cref{homsofweaklyap}. 
For essential surjectivity, it suffices to show that any object of $\aperf( \spec(R)
\setminus V(I))_{\geq 0}$ can be written as a geometric realization of a
simplicial object which at each level is a direct sum of copies of the unit. This
holds more generally for any quasi-affine spectral scheme (modeled on the
spectrum of a \emph{connective} $E_\infty$-ring). 
Indeed, let $\sF \in \aperf( \spec(R)
\setminus V(I))_{\geq 0}$. 
Then by \cite[Prop. 9.6.6.1]{SAG}, there is a 
map from a direct sum of copies of the unit to $\sF$ inducing a surjection on
$\pi_0$-sheaves. Since $\sF$ is almost perfect and connective, we can assume 
that this is a finite direct sum. Continuing inductively as in the proof of
\Cref{expressasgeometricreal}, we can now write $\sF$ as a geometric realization as desired. 
The assertions for $\mathcal{M}(R)$ follow from those for $\mathcal{M}(R)_{\geq
0}$. 
\end{proof}

\section{Review of monadicity and descent}

Our descent results take the 
form of an expression for a stable $\infty$-category as a homotopy limit of a
cosimplicial stable $\infty$-category. 
In this section, we review some general results from \cite[Sec.~4.7.5]{HA} for
identifying such homotopy limits; these are closely related to the
Barr--Beck--Lurie monadicity
theorem. 

\begin{definition}[The Beck--Chevalley condition] 

Let $\mathcal{C}^\bullet$ be an augmented cosimplicial $\infty$-category. We will say that $\mathcal{C}^\bullet$ \emph{satisfies the adjointability
condition} if for each $\alpha \cl [m] \to [n]$ in $\Delta^+$, the square
\begin{equation} \label{rightadjsquare}  \xymatrix{
\mathcal{C}^m \ar[d]^{\alpha}  \ar[r]^{d^0} &  \mathcal{C}^{m+1}
\ar[d]^{\alpha'}  \\
\mathcal{C}^n  \ar[r]^{d^0} &  \mathcal{C}^{n+1}
} \end{equation}
is right adjointable: that is, the horizontal arrows admit right adjoints
(denoted $d^0_*$), and
the induced Beck--Chevalley  transformation $\alpha \circ d^0_* \to d^0_* \circ
\alpha'$ is an equivalence. 
\end{definition} 

\begin{example} 
Let $R \to R'$ be a map of $E_\infty$-rings, and let $R'^\bullet$ be the associated
\v{C}ech nerve (an augmented cosimplicial $E_\infty$-ring). 
Then the diagram of $\infty$-categories $\mathrm{Mod}( R'^\bullet)$ satisfies
the adjointability condition. The reason is that if 
we have a pushout square of $E_\infty$-rings
\[ \xymatrix{
A \ar[d]  \ar[r] &  B \ar[d]  \\
A' \ar[r] &  B'
},\]
then the induced square after applying $\md(\cdot)$ is right adjointable. 
\label{modulesadj}
\end{example}

For the next result, we let $\mathrm{Cat}_\infty$ denote the $\infty$-category
of $\infty$-categories. 
\begin{theorem}[{Lurie, \cite[Corollary 4.7.5.3]{HA}}] 
\label{augcosimpthm}
Let $\mathcal{C}^\bullet \cl \Delta^+ \to \mathrm{Cat}_\infty$  
be a functor. Suppose that: 
\begin{enumerate}
\item The (co)augmentation $\mathrm{coaug} \cl \mathcal{C}^{-1} \to
\mathcal{C}^0$ is conservative and has the following property: totalizations of   
$\mathrm{coaug}$-split cosimplicial objects in $\mathcal{C}^{-1}$ exist and are
preserved by $\mathrm{coaug}$. 
\item 
$\mathcal{C}^\bullet$ satisfies the adjointability condition. 
\end{enumerate}
Then $\mathcal{C}^\bullet$ is a limit diagram. 
\label{augcosimp}
\end{theorem}

In practice, the adjointability condition will be automatic (from
e.g., \Cref{modulesadj}), so to verify that certain diagrams are limit diagrams, it
will be necessary to verify condition (1) of 
\Cref{augcosimpthm}. 
There will be two basic tools: universal descent maps (for which condition
(1) will hold for essentially diagrammatic reasons) and 
situations where one has a $t$-structure. 

We begin with the universal descent case. The following definition is
essentially from \cite[Sec.~3]{MGal} and  \cite[Sec.~D.3]{SAG}. 
The main example is given by base-change along a universal descent morphism of
$E_\infty$-rings. 

\begin{definition}[Universal descent functors] 
\label{univdescfun}
Let $\mathcal{C}, \mathcal{D}$ be idempotent-complete, stable
$\infty$-categories. 
We say that an exact functor $f \cl \mathcal{C} \to \mathcal{D}$ is of
\emph{universal descent} 
if there exists an exact functor $\phi \cl \mathcal{C} \to \mathcal{C}$ with the
following properties: 
\begin{enumerate}
\item The identity functor 
$\mathrm{id}_{\mathcal{C}}$ is a retract of $\phi$. 
\item There exists a filtration 
in the $\infty$-category of functors $\fun(\mathcal{C},
\mathcal{C})$, 
\begin{equation} \label{phifilt} \phi_0 \to \phi_1 \to \dots \to \phi_e = \phi
\end{equation} such that 
each $\phi_i/\phi_{i-1} \in \fun(\mathcal{C}, \mathcal{C})$ can
be written as $\psi_i \circ f$ for some $\psi_i \in
\fun(\mathcal{D}, \mathcal{C})$ exact. 
\end{enumerate}
\end{definition} 

\newcommand{\catex}{\mathrm{Cat}_\infty^{\mathrm{perf}}}
In the following, we denote by $\catex$ the $\infty$-category of
idempotent-complete, stable $\infty$-categories and exact functors between them. 

\begin{proposition} 
\label{univdescgeneralcrit}
Let $\mathcal{C}^\bullet \cl \Delta^+ \to \catex$
be a functor satisfying the adjointability condition. Suppose $\mathrm{coaug}\cl 
\mathcal{C}^{-1} \to \mathcal{C}^0$ is universal descent. Then
$\mathcal{C}^\bullet$ is a limit diagram. 
\end{proposition} 
\begin{proof} 
It remains to verify condition (1) of \Cref{augcosimp}. 
First, the filtration 
\eqref{phifilt} implies that any object in the kernel of $\mathrm{coaug}$
vanishes, so $\mathrm{coaug}$ is conservative. 
Next, consider the collection $\mathfrak{V}$ of cosimplicial objects $Y^\bullet \in \fun(\Delta,
\mathcal{C}^{-1})$
which admit a totalization in $\mathcal{C}^{-1}$ and which is preserved under $\mathrm{coaug}$. 
Clearly $\mathfrak{V}$ is a thick subcategory of 
$\fun(\Delta,
\mathcal{C}^{-1})$, and it contains those cosimplicial objects which admit
splittings. 
If 
$X^\bullet \in \fun(\Delta,
\mathcal{C}^{-1})$ is such that 
$\mathrm{coaug}(X^\bullet)$ admits a splitting, then 
our assumption implies that $X^\bullet$ belongs to $\mathfrak{V}$, whence the claim. 
\end{proof} 

Next, we prove a descent criterion in the case of $t$-structures. 
This is essentially a version of the argument used for faithfully flat descent
in \cite[Sec.~D.6]{SAG}.

\begin{proposition}[Descent criterion for left-complete stable
$\infty$-categories] 
\label{descleftcompl}
Let $\mathcal{C}^\bullet$ be an augmented cosimplicial object of
$\catex$. Suppose that each $\mathcal{C}^i, i \geq -1$ is equipped with a
right-bounded, left-complete $t$-structure,
and the following conditions are satisfied. 
\begin{enumerate}
\item 
 $\mathcal{C}^\bullet$ satisfies the
adjointability condition. 
\item Each  cosimplicial structure map is
right $t$-exact, and has a $t$-exact right adjoint. 
\item The coaugmentation functor $\mathrm{coaug} \cl \mathcal{C}^{-1} \to
\mathcal{C}^0$ is  conservative and $t$-exact. 
\end{enumerate}
Then $\mathcal{C}^\bullet$ is a limit diagram. 
Similarly for $\mathcal{C}_{\geq 0}^\bullet$ and $\mathcal{C}_{[m, n]}^\bullet$
for any $m \leq n$. 
\end{proposition} 
\begin{proof} 
For each $m \leq n$, the hypotheses of the proposition
yield (from $\mathcal{C}^\bullet$) an augmented cosimplicial $\infty$-category $\mathcal{C}^{\bullet}_{[m,
n]}$; the cosimplicial structure maps are obtained by truncating those of
$\mathcal{C}^\bullet$. 
Moreover, the hypotheses show that 
$\mathcal{C}^{\bullet}_{[m,
n]}$ satisfies the adjointability condition, and that $\mathcal{C}^{-1}_{[m, n]}
\to \mathcal{C}^0_{[m, n]}$ preserves totalizations and is conservative. 
Therefore, by {\Cref{augcosimp}}, we find that $\mathcal{C}^{\bullet}_{[m, n]}$
is a limit diagram. 
Taking the limit over $n$ and the colimit over $m$, we obtain that
$\mathcal{C}^\bullet$ is a limit diagram. 
A similar argument works for $\mathcal{C}^\bullet_{\geq 0}$. 
\end{proof}

\begin{example}[Derived faithfully flat descent] 
Let $A \to B$ be a faithfully flat map of commutative rings (or more generally
of connective $E_\infty$-rings). 
Then, via the Postnikov $t$-structures, we can apply \Cref{descleftcompl} to
obtain an equivalence (where $>-\infty$ denotes bounded-below objects)
\[ \md(A)_{> -\infty} \simeq \varprojlim \left( \md(B)_{> -\infty}
\rightrightarrows \md(B \otimes_A B)_{> -\infty}
\triplearrows \dots \right).\]
This is the derived version of faithfully flat descent
from \cite[Cor. D.6.3.3]{SAG} (at least in the bounded-below case), and the above argument is that of
\emph{loc.~cit}. 
\end{example}

\section{The universal descent case}

In this section, we show (\Cref{rigiddescuniv}) 
that almost perfect complexes 
on $\spec( \widehat{R}_I)\setminus V(IR)$ 
form a sheaf 
with respect to the universal descent topology (\Cref{univdescrings}). 
As an application, we prove $\arc$-descent results for the category of finite
\'etale covers, extending results of \cite{BM}. 

\subsection{The descent theorem}
\begin{definition}[Universal descent morphisms] 
\label{univdescrings}
Let $f \cl R \to R'$ be a morphism of $E_\infty$-rings. 
We will say that $f$ is a \emph{universal descent morphism}  if  the base-change functor $\md(R) \to \md(R')$ is a universal
descent functor in the sense of \Cref{univdescfun}. 
Alternatively, this holds if and only if the thick subcategory of $\md(R)$
generated by the $R'$-modules is all of $\md(R)$. 
Using this notion, we obtain the \emph{universal descent topology} on the
opposite of the $\infty$-category of $E_\infty$-rings. 

\end{definition}

\begin{remark} 
\label{twooutofthree}
The class of universal descent morphisms is closed under composition. Moreover,
if a composite $R \to S \to T$ of morphisms of $E_\infty$-rings is a universal descent morphism, then $R \to S$
is a universal descent morphism. 
\end{remark}

\begin{definition}[Exponents of universal descent morphisms]
Suppose $f \cl R \to R'$ is a universal descent morphism. 
Let $R'^\bullet$ be the \v{C}ech nerve. Then $R \simeq \mathrm{Tot}(
R'^\bullet)$. Moreover, the associated $\mathrm{Tot}$-tower
$\left\{\mathrm{Tot}^n(R'^\bullet)\right\}$ defines a constant pro-object
of $\md(R)$, so there exists $e \geq 0$ such that
$R$ is a retract of the partial totalization $\mathrm{Tot}^e(R'^{\bullet})$; the
smallest such $e$ is called the \emph{exponent}. 
\end{definition}

This class of morphisms was studied in \cite{MGal} and \cite[Sec.~D.3]{SAG}, to
which we refer for more details; see also \cite{Msurvey} for a survey. 
In particular, one has the following basic result (a special case of
\Cref{univdescgeneralcrit}): 
\begin{theorem} 
Given a universal descent morphism $R \to R'$, 
the natural map $\md(R) \to \varprojlim ( \md(R') \rightrightarrows \md(R'
\otimes_R R') \triplearrows \dots )$ is an equivalence. 
In other words, $\md(\cdot)$ is a sheaf for the universal descent topology. 
\end{theorem} 

In the case of finitely presented morphisms of discrete rings, one can give a
concrete geometric criterion for a map to be universal descent, due to Bhatt--Scholze \cite[Prop. 11.25]{BS15}
in the noetherian case; here we observe that it holds generally. 

\begin{theorem}[{}] 
\label{hcoverdesc}
Let $f: R \to R'$ be a finitely presented 
map of discrete rings. Suppose $f$ is a 
$v$-cover 
(\cite{Rydh} and \cite[Sec.~2]{BS15}). Then $f$ is a universal descent morphism. 
\end{theorem}  
\begin{proof} 
By \cite[Theorem 6.4]{Rydh}, it follows that $f$ is obtained 
as the (underived) base-change of a  $v$-cover of finitely presented
$\mathbb{Z}$-algebras, $R_0 \to R_0'$ along a map $R_0 \to R$. 
In particular, $R' = \tau_{\leq 0}(R \otimes_{R_0} R_0')$. 
The map $R_0 \to R_0'$ is a universal descent morphism \cite[Prop. 11.25]{BS15}. 
Now the map (of $E_\infty$-rings) $R \to R \otimes_{R_0} R_0'$ is universal
descent by base-change. 
By \Cref{discreteunivdesc}, 
it follows that $R \to R'$ is a universal descent morphism. 
\end{proof} 

\begin{lemma} 
\label{discreteunivdesc}
Let $f: R \to S$ be a map of connective $E_\infty$-rings which is universal
descent. 
Suppose $R$ is discrete. Then the composite map $R \to S \to \pi_0(S)$ is
universal descent. 
\end{lemma} 
\begin{proof} 
We use throughout the following basic observation, valid since $R$ is discrete. Let $R \to M$ be a map of
$R$-modules; then this map admits a section if and only if the composite $R \to
M \to \tau_{\leq 0} M$ admits a section. 

Let $S^\bullet$ be  \v{C}ech nerve of $R \to S$. 
Then for some $N$, $R$ is a retract of $\mathrm{Tot}^N( S^\bullet)$;
equivalently, the map $R
\to \mathrm{Tot}^N( S^\bullet) \to \tau_{\leq 0}\left( \mathrm{Tot}^N(
S^\bullet) \right)
$ admits a section. 
Now the map 
$\tau_{\leq 0}\left( \mathrm{Tot}^N(
S^\bullet) \right)
 \to \tau_{\leq 0}\left( \mathrm{Tot}^N(
\tau_{\leq N+1}(S^\bullet)) \right)$ is an equivalence. 
Therefore, the map 
$R \to \tau_{\leq 0}(\mathrm{Tot}^N( \tau_{\leq N+1} (S^{\bullet})))$ admits a
section, and hence the map 
$R \to \mathrm{Tot}^N( \tau_{\leq N+1} (S^{\bullet}))$ admits a section too. 
It follows that $R$ is a retract of a finite limit of a diagram of
$R$-modules, each of which admits the structure of 
a $\tau_{\leq N+1} S$-module. Consequently $R \to \tau_{\leq N+1} S$ is a
universal descent morphism. 
Via the Postnikov tower, we see that $\tau_{\leq N+1} S \to \pi_0(S)$ is a
universal descent morphism. Composing, the claim follows. 
\end{proof} 

\begin{proposition} 
\label{completeIdescent}
Let $R \to R'$ be a universal descent map of $E_\infty$-rings, and let $I
\subset \pi_0(R)$ be a finitely generated ideal. Then $\hat{R} \to \hat{R'}$ is
universal descent. 
\end{proposition} 
\begin{proof} 
This follows because we can write $R$ as a retract of a finite limit of a
diagram of $R$-modules admitting the structure of $R'$-module, and then complete
everywhere. 
\end{proof} 

The main descent theorem 
that we prove in this section is the following. 

\begin{theorem}[Rigid analytic descent in the universal case] 
\label{rigiddescuniv}
Let $ R $ be a connective $E_\infty$-ring and let $I
\subset \pi_0(R)$ be a finitely generated ideal. Then, on the $\infty$-category
of 
connective $E_\infty$-$R$-algebras, the construction $R' \mapsto \aperf( \spec(
\widehat{R'}_I) \setminus V(I))$ is a sheaf for the universal descent topology. 
Similarly for $R' \mapsto \aperf(\spec(
\widehat{R'}_I) \setminus V(I))_{\geq 0}$ and with $\perf$ replacing $\aperf$. 
\end{theorem} 

The proof of 
\Cref{rigiddescuniv} is based on showing that $R' \mapsto \mathcal{M}(R')$ is a
sheaf for the universal descent topology, and that the property of an object in
$\mathcal{M}(R')$ being weakly almost perfect is local.

\begin{proposition} 
\label{descentM1}
The construction $R' \mapsto \mathcal{M}(R')$ is a sheaf for the
universal descent topology. 
\end{proposition} 
\begin{proof} 
Given a universal descent map $S \to S'$ of connective $E_\infty$-$R$-algebras, we form the \v{C}ech nerve
$S^\bullet$, an augmented cosimplicial ring. The augmented cosimplicial stable
$\infty$-category $\mdc{S^\bullet}$ satisfies the adjointability condition;
therefore, so does $\mathcal{M}(S^\bullet)$ since the relevant adjoints pass
through the procedure that constructs $\mathcal{M}(S)$ from $\mdc{S}$. 
Now $\mathcal{M}(S) \to \mathcal{M}(S')$ is a universal descent functor, by our
assumption that $S \to S'$ is universal descent as a map of $E_\infty$-rings. Therefore, the result follows
from \Cref{univdescgeneralcrit}.  
\end{proof}

\begin{proposition} 
Let $R \to R'$ be a universal descent map of connective $E_\infty$-rings of
exponent $e$. 
Let $M \in \md(R)$ be bounded-below. 
Suppose $R' \otimes_R M $ is $\leq I^r $-perfect to order $n$. 
Then $M $ is ${\leq I^{re}}$-perfect to order $n - e-1$. 
\label{descentstuff}
\end{proposition} 
\begin{proof} 
For any filtered category $\I$ and functor $\I \to \md(R)_{\leq n-e-1}$, $j \mapsto N_j$, we
will verify that
the natural map 
\begin{equation}  \label{naturalmapfilteredcolim} \omega \cl   \varinjlim_J
\tau_{\geq 0}\hom_R(M, N_j) \to \tau_{\geq 0}\hom_R( M, \varinjlim_J
N_j)  \end{equation}  
is an $\leq I^{re}$-isogeny (compare \Cref{critforIperfectn}). 
More generally, for each functor $f \cl \I \to \md(R)$, we consider the map 
$\omega= \omega(f)$ of 
\eqref{naturalmapfilteredcolim}. 
Our hypothesis implies that if $f$ lifts to a functor with values in $\md(R')_{\leq
n-1}$, then
$\omega(f)$ is a $\leq I^r$-isogeny. 

Let $R'^\bullet$ be the \v{C}ech nerve of $R \to R'$, considered as a
cosimplicial $R$-module. 
For any $R$-module $N$, we can write $N$ as a functorial retract of 
$\mathrm{Tot}^e( N \otimes_R R'^{\bullet})$. 
It follows that if $N \in \md(R)_{\leq n-e-1}$, then 
we can write $N$ as a functorial retract of 
$\mathrm{Tot}^e( \tau_{\leq n-1} (N \otimes_R R'^{\bullet}))$.

Now returning to \eqref{naturalmapfilteredcolim} in the case where
$f$ takes values in $\leq n-e-1$-truncated modules, 
we find that
$\omega(f)$ is a retract of a partial totalization $\mathrm{Tot}^e$ of 
$\omega( \tau_{\leq n-1}(f \otimes_R R'^\bullet))$. 
From the above, it follows that $\omega(f)$ is a retract of a partial 
totalization $\mathrm{Tot}^e$ of a diagram of $\leq I^r$-isogenies, whence the
claim. 
\end{proof}

\begin{proposition} 
\label{localdescRperfIinfinity}
Let $R \to R'$ be a universal descent map of $I$-complete connective $E_\infty$-rings of
exponent $e$. 
Let $M \in \md(R)$ be $I$-complete and bounded-below. 
Suppose $\widehat{R' \otimes_R M} \in \md(R')$ is $\leq I^\infty$-perfect to order $n$. 
Then $M$ is $\leq I^\infty$-perfect to order $n-e-1$. 
\end{proposition} 
\begin{proof} 
Fix a tower $\left\{R_m\right\}$ as in \Cref{towercompletion}. 
By \Cref{perfecttoordermI}, 
it suffices to show that there exists $r$ such that $M \otimes_R R_m$ is $\leq
I^r$-perfect to order $n-e-1$ for all $m$. 
But by assumption, there exists $r'$ such that 
$(M \otimes_R R_m) \otimes_R R'$ is $\leq I^{r'}$-perfect to order $n$ for all
$m$. 
Now apply \Cref{descentstuff} (with $r = r'e$) to conclude. 
\end{proof}

\begin{corollary} 
\label{cor:detectperfect}
Let $R \to R'$ be a universal descent map of $I$-complete connective $E_\infty$-rings of
exponent $e$. 
Let $M \in \mathcal{M}(R)$. If $M \otimes_R R'$ is weakly perfect to order $n$, then
$M$ is weakly perfect to order $n-e-1$. 
\end{corollary}
\begin{proof} 
This follows from \Cref{localdescRperfIinfinity}, since we can (up to truncating
homotopy groups in high enough degrees) assume that $M$ is represented by an
object of $\mdc{R}$. 
\end{proof}

\begin{proof}[Proof of \Cref{rigiddescuniv}]
We first prove the results for $\perf, \aperf$. 
Since $\perf$ is the subcategory of dualizable objects in $\aperf$, it suffices
to prove the result for $\aperf$.
We have seen that the construction 
$R' \mapsto \mathcal{M}(R')$ is a sheaf for the universal descent topology
(\Cref{descentM1}), 
and there is a natural
fully faithful embedding $\aperf( \spec( \hat{R'}) \setminus V(I)) \subset
\mathcal{M}(R')$ (\Cref{jgivesequiv}) with image the weakly almost perfect objects.
It suffices now to show that the property of belonging to the image of this 
embedding (i.e., being weakly almost perfect) is local in the universal descent topology. 
But this follows from \Cref{cor:detectperfect}. 
Note that it suffices to work everywhere with $I$-complete $E_\infty$-rings
(\Cref{completeIdescent}). 

To obtain the results for $\aperf_{\geq 0}$ (and hence $\perf_{\geq
0}$), it suffices to show that for a universal descent
map $R' \to S'$ 
of $E_\infty$-$R$-algebras and an object $M \in \aperf( \spec( \hat{R'}
\setminus V(I)))$ whose pullback to 
$\aperf( \spec( \hat{S'} \setminus V(I))$ is connective, then $M$ is connective. 
This is a local question, 
and $\hat{R'} \to \hat{S'}$ is of universal descent. 
Thus, it suffices to verify
the following: if $A \to B$ is a universal descent map of connective 
$E_\infty$-algebras
and $M \in \md(A)$ is almost perfect, then $M$ is connective if and only if $M
\otimes_A B$ is connective. 
By induction, we can assume $M$ is $(-1)$-connective. But then $\pi_{-1}(M)$ is
a finitely presented $\pi_0(A)$-module and $\pi_{-1}(M) \otimes_{\pi_0(A)}
\pi_0(B) = 0$. Since $\pi_0(A) \to \pi_0(B)$ is surjective on spectra, this
forces (by Nakayama) $\pi_{-1}(M) = 0$ as desired. 
\end{proof} 
\subsection{$\arc$-descent}

Let $R$ be a $\mathbb{Z}[x_1, \dots, x_n]$-algebra. Our goal in this section is
to study the descent properties of the functor
\begin{equation}  \label{basicfunctor} R \mapsto \sF(R)
\stackrel{\mathrm{def}}{=}\left\{\text{finite \'etale schemes over } \spec(
\hat{R}_{(x_1, \dots, x_n)} \setminus V(x_1, \dots, x_n))\right\}  ,
\end{equation}
as a functor from 
$\mathbb{Z}[x_1, \dots, x_n]$-algebras to categories. 

\begin{remark}[Derived completions versus completions] 
In \eqref{basicfunctor}, as usual, the notation $\widehat{R}_{(x_1, \dots,
x_n)}$ refers to the derived completion (which is a connective $E_\infty$-ring). 
One can replace the derived completion 
with the classical completion. 
For any $R$, there is a map from the derived completion 
$\widehat{R}_{(x_1, \dots, x_n)}$  to the classical completion
$\hat{R}^{\mathrm{cl}}_{(x_1, \dots, x_n)}$. 
The induced map on $\pi_0$ is surjective, and the kernel squares to zero; this
follows from the $\varprojlim$-spectral sequence (i.e., the Milnor exact
sequence in this case).
Alternatively, this fact follows directly from \Cref{compositederivedtcomplete}. 
Consequently, finite \'etale schemes are the same
whether one uses the classical or derived completion. 
Similarly, in this section, there is no extra generality gained by working with
connective $E_\infty$-rings rather than discrete rings. 
\end{remark} 

\newcommand{\fet}{\mathrm{FEt}}
We use the following fundamental algebrization result, which allows one 
to replace the completion with the henselization. 
For a qcqs scheme $X$, we let $\fet(X)$ be the category of finite \'etale
schemes over 
$X$. 

\begin{theorem}[{Elkik \cite{Elkik}, Gabber--Ramero \cite[Prop. 5.4.54]{GR}, Gabber
\cite[Th. 2.1.2]{Gabbertravaux}}] 
\label{EGR}
For any 
$\mathbb{Z}[x_1, \dots, x_n]$-algebra 
$R$  which is henselian along $(x_1, \dots, x_n)$, the natural functor
$$ \fet( \spec(R) \setminus V(x_1, \dots, x_n)) \to \fet( \spec(\hat{R}_{x_1,
\dots, x_n}) \setminus V(x_1, \dots, x_n)) $$
is an equivalence of categories. 
\end{theorem} 

\begin{corollary} 
\label{Fisfinitary}
The functor $\sF$ is finitary, i.e., commutes with filtered colimits. 
\end{corollary}
\begin{proof} 
This follows from \Cref{EGR}, since henselization commutes with filtered
colimits. 
\end{proof} 

We now use the following definition, from \cite{BM} and due independently
to Rydh. 
It is a refinement of the $v$-topology \cite{Rydh, BS15}, which in turn
is the non-noetherian version of Voevodsky's $h$-topology. 
We also note a slight variant of it, as in \cite[Sec.~6.2]{BM}. 
\begin{definition}[$\arc$-covers] 
A map of commutative rings $R \to R'$ is said to be an \emph{$\arc$-cover}
if for every rank $\leq 1$-valuation ring $V$ and map $R \to V$, there is an
extension of rank $\leq 1$-valuation rings $V \to V'$ and a commutative diagram
\begin{equation} \label{arccommdiag} \xymatrix{
R \ar[d]  \ar[r] &  R' \ar[d] \\
V \ar[r] &  V'
}.\end{equation}
This defines the \emph{$\arc$-topology} 
on the category of affine schemes. 
\end{definition} 

\begin{definition}[{The $\arc_{(x_1, \dots, x_n)}$-topology}]
A map $R \to R'$ of 
$\mathbb{Z}[x_1, \dots, x_n]$-algebras is said to be an $\arc_{(x_1, \dots,
x_n)}$-cover if for every rank $\leq 1$ valuation ring $V$ with map $R \to
V$ such that 
the image of $(x_1, \dots, x_n)$ in $V$ is nonzero but contained in the maximal
ideal, there exists an extension of rank $\leq 1$-valuation rings $V \to V'$ 
and a commutative diagram
as in \eqref{arccommdiag}. 
This defines the \emph{$\arc_{(x_1, \dots, x_n)}$-topology} on
the category of affine schemes over $\spec \mathbb{Z}[x_1, \dots, x_n]$. 
\end{definition}

\newcommand{\sG}{\mathcal{G}}
\begin{remark} 
\label{arcvsarct}
A map $R \to R'$ is an $\arc_{(x_1, \dots, x_n)}$-cover if and only if 
$R \to R' \times R/(x_1, \dots, x_n) \times R[1/x_1] \times \dots \times
R[1/x_n]$ is an $\arc$-cover. 
Therefore, a functor $\sG$ is a sheaf for the $\arc_{(x_1, \dots, x_n)}$-topology if
and only if it is a sheaf for the $\arc$-topology and 
$\sG(R/(x_1, \dots, x_n)), \sG(R[1/x_i])$ are the terminal object for any
$\mathbb{Z}[x_1, \dots, x_n]$-algebra $R$. \end{remark} 

We refer to \cite[Ex.~6.15]{BM} for a discussion 
of the relationship between the $\mathrm{arc}_x$-topology (so $n = 1$ in the
above) and surjectivity at the level of adic spectra. 

\begin{theorem} 
\label{arcsheafthm}
The functor $R \mapsto \sF(R)$ is a sheaf for the $\arc_{(x_1, \dots,
x_n)}$-topology. 
\end{theorem} 

Informally, \Cref{arcsheafthm} 
is an expression of a (well-known) principle
that the theory of (purely algebraically defined) finite \'etale schemes 
behaves well in analytic geometry, and satisfies very strong descent results;
many cases of this result are already in the literature. 
For instance, for perfectoid spaces, $v$-descent of finite \'etale
schemes appears as \cite[Prop. 9.7]{Diamonds}. 
In the case of \emph{abelian} \'etale covers, this result appears (and more
generally for the higher cohomology) in \cite[Cor. 6.17]{BM} (with the
restriction to a principal ideal, but this is not necessary). 
\Cref{arcsheafthm} as stated can also be proved using Gabber's rigidity results for
nonabelian cohomology, cf.~\cite{Gabberaffine} and
\cite[Exp.~XX]{Gabbertravaux}. 

We explain here a quick proof of \Cref{arcsheafthm} 
using the theory of finite \'etale algebra objects in a symmetric monoidal
stable $\infty$-category (\cite{RognesGal, MGal}) and \Cref{rigiddescuniv}.

\begin{construction}[Finite \'etale algebra objects]
\label{fetcategory}
For a small, idempotent-complete stably symmetric monoidal $\infty$-category $\mathcal{C}$, one
extracts 
\cite[Def. 6.1]{MGal}
a category $\fet( \mathcal{C})$ of ``finite \'etale algebra objects''
of $\mathcal{C}$, a full subcategory of the $\infty$-category
$\mathrm{CAlg}(\mathcal{C})$ of commutative algebra objects of
$\mathcal{C}$.\footnote{The definition in \emph{loc.~cit.} is stated for a 
presentably symmetric monoidal $\infty$-category; however, we can embed
$\mathcal{C}$ into $\mathrm{Ind}(\mathcal{C})$.} 
In the case where $\mathcal{C} = \perf(X)$ for $X$ a qcqs spectral scheme
(modeled on the spectra of connective $E_\infty$-rings), then $\fet(\mathcal{C})$ recovers precisely the
opposite to the category of finite \'etale schemes over $X$ (or equivalently of the underlying
scheme $\pi_0 X$). 

\end{construction}

\begin{example}[Torsors in $\fet(\mathcal{C})$] 
Let $G$ be a finite group. 
A $G$-torsor in $\fet(\mathcal{C})$ is
given by an object $A \in \fun(BG, \mathrm{CAlg}(\mathcal{C})$ 
(i.e., a commutative algebra equipped with a $G$-action) such that: 
\begin{enumerate}
\item  As an object of $\fun(BG, \mathcal{C})$, $A$ belongs to the thick
subcategory generated by the induced $G$-objects (i.e., the $G$-action is
nilpotent; compare the discussion in \cite[Sec.~4]{Msurvey}). 
In particular, $A^{hG}$ exists in $\mathcal{C}$. 
\item  The natural map $\mathbf{1} \to A^{hG}$ is an equivalence. 
\item The shearing map $A \otimes A \to \prod_G A$ is an equivalence in
$\mathcal{C}$. 
\end{enumerate}
These conditions are due to Rognes \cite{RognesGal}, who introduces the notion
of a Galois extension of an $E_\infty$-ring. 
\end{example}

\begin{remark}[$\fet(\cdot)$ preserves some limits] 
\label{fetpreserveslimits}
In general, the construction $\mathcal{C} \mapsto \fet(\mathcal{C})$ need not
preserve limits: the construction $\mathcal{C} \mapsto
\mathrm{CAlg}(\mathcal{C})$ does, but condition (1) involved may not.
However, suppose we have an augmented cosimplicial stably symmetric monoidal
$\infty$-category $\mathcal{C}^\bullet$ which is a limit diagram; suppose
moreover that $\mathcal{C}^{-1} \to \mathcal{C}^0$ is a universal descent
morphism. 
Then $\fet(\mathcal{C}^{-1}) \simeq \varprojlim(\fet( \mathcal{C}^\bullet))$. 
This follows by considering $G$-torsors for each finite group $G$ and
\cite[Cor. 5.40]{MGal}. 
\end{remark}

\begin{proof}[Proof of \Cref{arcsheafthm}] 
By \Cref{arcvsarct}, it is sufficient to verify that $\sF$ is an $\arc$-sheaf. 
We will apply the criterion of \cite[Theorem 4.1]{BM} to the functor $\sF$. 
First, note that the functor $\sF$ is finitary thanks to \Cref{Fisfinitary}. 

We show that $\sF$ satisfies $h$-descent. 
Let $R \to R'$ be an $h$-cover, i.e., a finitely presented $v$-cover. It suffices to show that in this case,
\begin{equation} \label{sheafproperty} \sF(R) \simeq \varprojlim( \sF(R')
\rightrightarrows \sF(\pi_0(R'
\otimes_R R')) \triplearrows \dots ).\end{equation} 
Indeed, by 
\Cref{hcoverdesc}, $R \to R'$ is a universal descent morphism, so by \Cref{rigiddescuniv} we have a limit
diagram 
\[ 
\perf( \spec( \widehat{R}_I) \setminus V(I)) \simeq \varprojlim \left( 
\perf( \spec( \widehat{R'}_I) \setminus V(I))
\rightrightarrows
\perf( \spec( \widehat{R' \otimes_R R'}_I) \setminus V(I)) \dots \right)
\]
(involving the iterated \emph{derived}
tensor products of $R'$ over $R$).
Applying 
\Cref{fetcategory}, and noting that the higher homotopy groups in the derived
tensor product do not affect finite \'etale schemes, \eqref{sheafproperty}
follows (also via \Cref{fetpreserveslimits}). 

To check the criterion of 
\cite[Theorem 4.1]{BM}, it suffices to check the ``aic-v-excision'' condition
of \emph{loc.~cit}.  
Suppose $V$ is an absolutely integrally closed valuation ring and 
$\mathfrak{p} \subset V$ a prime ideal; we need to show that the square
\begin{equation} \label{Vpullback} \xymatrix{
\sF( V) \ar[d]  \ar[r] &  \sF(V/\mathfrak{p}) \ar[d]  \\
\sF(V_{\mathfrak{p}}) \ar[r] &  \sF( (V/\mathfrak{p})_{\mathfrak{p}})
}\end{equation}
is a pullback. To see this, we consider the map 
$V \to V/\mathfrak{p} \times V_{\mathfrak{p}}$, which is of universal descent
because of the pullback description $V \simeq V/\mathfrak{p}
\times_{(V/\mathfrak{p})_{\mathfrak{p}}} V_\mathfrak{p}$. 
It follows as in the previous paragraph that $\sF$ carries the \v{C}ech nerve
of this map to a limit diagram; unwinding the \v{C}ech nerve now shows that 
\eqref{Vpullback} is a pullback square as desired. 
\end{proof} 

\section{Example: flat descent on classical rigid spaces}
Throughout this section, we fix a complete nonarchimedean field $K$ with ring of
integers $\mathcal{O}_K \subset K$, and a nonzero nonunit $\pi \in
\mathcal{O}_K$. We will work in the setting of classical rigid geometry and
$K$-affinoid algebras.\footnote{See also work of Ducros \cite{Ducros}
extending the notion of flatness 
to Berkovich analytic spaces, which we will not consider.}
In this section, we give a proof of \Cref{aperfrigiddesc} from the introduction
(as \Cref{aperfrigiddesc2} below), i.e., of flat descent of almost perfect
complexes for $K$-affinoid algebras. 
This result  can also be deduced from the case of coherent sheaves \cite{BG},
although our argument will be independent; 
cf.~also \cite{HPV} for the case of \'etale descent.

We freely use the flattening results of Bosch--L\"utkebohmert \cite{BL} and
various finiteness properties of topologically finitely presented
$\mathcal{O}_K$-algebras, but
otherwise the methods make no reference to results of rigid geometry (e.g.,
Kiehl's theorem that coherent modules satisfy descent in the analytic
topology; these methods recover Kiehl's theorem). 
Our main observation 
 (\Cref{fflatgeneric})
is that if $A \to B$ is a faithfully flat map of affinoid
$K$-algebras, then the map $A_0 \to B_0$ of appropriate rings of definition is a
universal descent morphism; 
furthermore, the iterated derived tensor products of $B_0$ over $A_0$ have
higher homotopy which is bounded $\pi$-power torsion. 
The result will then follow from \Cref{rigiddescuniv}.

\subsection{Review of coherent rings}

In the following it will be necessary to work with coherent (and especially
stably coherent) rings; these are rings for which many of the convenient
module-theoretic finiteness properties of noetherian rings still hold, provided
one restricts to finitely presented modules. 
See \cite{Glaz} for a textbook reference. 
\begin{definition}[Coherent rings] 
A ring $R$ is called \emph{coherent} if the category of finitely presented
$R$-modules is abelian (equivalently, stable under kernels). 
This holds if and only if each finitely generated ideal $I \subset R$ is
finitely presented as an $R$-module. 
A commutative ring $R$ is called \emph{stably coherent} if 
every finitely presented $R$-algebra is coherent. 
\end{definition}

\begin{definition}[Coherent $E_\infty$-rings] 
Let $R$ be a connective $E_\infty$-ring. 
We say that $R$ is \emph{coherent} (resp. \emph{stably coherent}) if $\pi_0(R)$ 
is coherent (resp. stably coherent) as a commutative ring and each $\pi_i(R), i \geq 0$
is finitely presented as a $\pi_0(R)$-module. 
\end{definition} 

\begin{remark}[Characterization of almost perfect complexes] 
Suppose $R$ is a coherent $E_\infty$-ring. 
In this case, a bounded-below $R$-module
spectrum $M \in \md(R)$ is almost perfect if and only if the homotopy groups
$\pi_i(M)$ 
are finitely presented $R$-modules, cf.~\cite[Prop.~7.2.4.17]{HA}. 
In particular, $\aperf(R)$ acquires a $t$-structure (by restriction from
$\md(R)$) whose heart is the category of finitely presented discrete
$\pi_0(R)$-modules; this $t$-structure is right-bounded and left-complete. 
\end{remark}

Let $R$ be a coherent $E_\infty$-ring. 
There is similarly a characterization of almost finitely presented
$E_\infty$-$R$-algebras.
Given any connective $E_\infty$-ring $R$, recall  \cite[Sec.~7.2.4]{HA}
that a connective
$E_\infty$-$R$-algebra $R'$ is said to be \emph{almost finitely presented} if,
for each $n$, 
the truncation $\tau_{\leq n} R'$ defines a  compact object of the
$\infty$-category of connective, $n$-truncated $E_\infty$-$R$-algebras.

\begin{proposition}[Characterization of almost finitely presented algebras] 
\label{afpalgebra}
Let $R$ be a stably coherent $E_\infty$-ring. 
Let $R'$ be a connective $E_\infty$-$R$-algebra. Then the following are
equivalent: 
\begin{enumerate}
\item $R'$ is almost finitely presented.  
\item The ring $\pi_0(R')$ is finitely presented as a $\pi_0(R)$-algebra.
Moreover, $R'$ is coherent as an $E_\infty$-ring. 
\end{enumerate}
In this case, $R'$ is stably coherent as well. 
\end{proposition} 
\begin{proof} 
The proof of \cite[Prop. 7.2.4.31]{HA} (in the noetherian case) works equally here. 
\end{proof} 

\begin{remark} Coherence will be useful for us in the following situation. 
Let $R$ be an $E_\infty$-$\mathcal{O}_K$-algebra  which is coherent. 
Suppose that $R[1/\pi]$ is discrete. Then 
for each $i>0$, the homotopy group
$\pi_i(R)$ is $\leq \pi^\infty$-isogenous to zero. 
This follows because, for $i > 0$, $\pi_i(R)$ is a finitely presented module over $\pi_0(R)$
and $\pi_i(R)[1/\pi] = 0$ by assumption. 
\label{useofcoherence}
\end{remark}

\begin{example} 
The stable coherence of
valuation rings (such as $\mathcal{O}_K$) is  a consequence of the results of  
Raynaud--Gruson \cite{RG} (see, e.g., \cite[Theorem
7.3.3]{Glaz}). 
That is, the polynomial ring $\mathcal{O}_K[T_1,\dots, T_n]$ is coherent for any
$n$. 
The coherence of 
the ring $\mathcal{O}_K\left \langle T_1, \dots, T_n\right\rangle
\stackrel{\mathrm{def}}{=}
\widehat{\mathcal{O}_K[T_1, \dots, T_n]}_{\pi}
$
can be deduced similarly; see \cite[Sec.~7.3]{Bosch} for an account. 

\end{example} 
We will need the following strengthening of the above example, which appears in work of Fujiwara--Gabber--Kato \cite[Prop.
4.3.4, Prop. 7.2.2, and Th. 7.3.2]{FGK}. 
See also \cite{Katosurvey} for a survey. 

\begin{theorem}[] 
\label{main:finitenessthm}
Let $K$ be a complete nonarchimedean field. 
Let $R$ be a finitely presented algebra over 
the $\pi$-completed polynomial ring $\mathcal{O}_K\left \langle T_1, \dots, T_n\right\rangle$. Then: 
\begin{enumerate}
\item  
$R$ is coherent. 
\item
Given any finitely generated $R$-module $M$, the
$\pi$-power torsion submodule is finitely generated. 
Any $\pi$-torsion-free finitely generated module is finitely presented. 
\item The map from $R$ to its $\pi$-adic completion (which is also the derived
$\pi$-adic completion by (2)), $R \to \hat{R}$,  is flat. 
\end{enumerate}
\end{theorem}

In particular, the result implies that for 
finitely presented 
algebras over $\mathcal{O}_K\left \langle T_1, \dots, T_n\right\rangle$ and
finitely generated modules, there is no distinction between classical and
derived completion.

\subsection{The flat topology on $K$-affinoid algebras}
Here we review some facts about flatness for maps of $K$-affinoid
algebras. For a detailed treatment, see \cite[Ch. 5]{EGR}. 

\newcommand{\affn}{\mathrm{Affinoid}}
\newcommand{\tft}{\mathrm{Alg}^{\mathrm{tfp}}}
\begin{definition} 
\begin{enumerate}
\item  
An $\mathcal{O}_K$-algebra $A_0$ is \emph{topologically of finite presentation} if 
$A_0$ is a quotient of some $\mathcal{O}_K \left \langle T_1, \dots,
T_n\right\rangle$ by a finitely generated ideal; if $A_0$ is torsion-free, it
suffices that $A_0$ is a quotient of 
some $\mathcal{O}_K \left \langle T_1, \dots,
T_n\right\rangle$. 
Let $\tft_{\mathcal{O}_K}$ denote the category of  $\mathcal{O}_K$-algebras
which are topologically of 
finite presentation. 
We let $\tfft_{\mathcal{O}_K} \subset \tft_{\mathcal{O}_K}$ be the subcategory
of those algebras which are flat (i.e., torsion-free) over $\mathcal{O}_K$. 
\item
An \emph{$K$-affinoid algebra} is a $K$-algebra which is a quotient of 
the Tate algebra $T_n = K\left \langle T_1, \dots, T_n
\right\rangle
\stackrel{\mathrm{def}}{=}
(\mathcal{O}_K \left \langle T_1, \dots, T_n\right\rangle)[1/\pi]$. 
Let $\affn_K$ denote the category of $K$-affinoid algebras. 
We have an essentially surjective functor $\tft_{\mathcal{O}_K} \to \affn_K$ given by tensoring
with $K$. 
\end{enumerate}
The categories $\tft_{\mathcal{O}_K}, \affn_K$ admit finite colimits via the
completed tensor products; the natural functor
$\tft_{\mathcal{O}_K} \to \affn_K$ preserves them. 
\end{definition} 

\begin{definition}[Flat morphisms in $\tfft_{\mathcal{O}_K}, \affn_K$] 
We define a morphism in $\tfft_{\mathcal{O}_K}, \affn_K$ 
to be \emph{flat} (resp.~\emph{faithfully flat}) if it is flat (resp.~faithfully
flat) as a map of ordinary rings. 
This condition is stable under (completed) base-change: 
\begin{enumerate}
\item  
In $\tfft_{\mathcal{O}_K}$, a map $A_0 \to B_0$ is flat (resp.~faithfully flat) if and only if
it is flat (resp.~faithfully flat)
modulo each power of $\pi$. 
In the case of flatness, this follows because the flatness condition can be tested on coherent 
(discrete) $A_0$-modules, and these are automatically $\pi$-complete. 
Faithful flatness is implied by flatness and universal descent, so the claim for
faithful flatness follows from \Cref{picompletenesslemma} below. 
\item
The statement that flat (resp.~faithfully flat) maps in $\affn_K$ are stable under
base-change 
follows from the existence of flat (resp.~faithfully flat) formal
models (\Cref{flattening} below).
\end{enumerate}
\end{definition}

\begin{example} 
\begin{enumerate}
\item  
Let $A_0 \to B_0$ be a flat map in $\tfft_{\mathcal{O}_K}$. Then
$A_0[1/\pi] \to  B_0[1/\pi]$ is a flat map
of $K$-affinoid algebras. 
\item
Let $A_0 \in \tft_{\mathcal{O}_K}$ and $A_0'$ be a finitely presented
$A_0$-algebra. 
Suppose $A_0[1/\pi] \to A'_0[1/\pi]$ is flat. 
Then the map $A_0[1/\pi] \to \widehat{A'_0}[1/\pi]$ of $K$-affinoid algebras is
flat. 
Indeed, this follows because $A_0' \to \widehat{A'_0}$ is flat
(\Cref{main:finitenessthm}). 
\end{enumerate}
\end{example} 

One can ask when a map in $\tfft_{\mathcal{O}_K}$ has flat generic fiber. 
It turns out that the above two are the essential cases. 
To see this, we will use the flattening results of \cite{BL}, after 
the work of Raynaud-Gruson \cite{RG} in the case of schemes. 
See also \cite[Sec.~5.8]{EGR} for an account. 
The last assertion (of faithful flatness rather than flatness) 
follows from \cite[Prop. 5.5.10]{EGR}. 
\newcommand{\XX}{\mathfrak{X}}
\newcommand{\YY}{\mathfrak{Y}}
\begin{theorem}[Bosch--L\"{u}tkebohmert \cite{BL}] 
\label{flattening}
Let $A_0 \to B_0$ be a map in $\tfft_{\mathcal{O}_K}$. 
Let $\XX = \spf A_0, \YY = \spf B_0$ and let $f: \YY \to \XX$ be the induced map. 
Suppose the map $A_0[1/\pi] \to B_0[1/\pi]$ is flat.
There exists a 
diagram of formal schemes 
\[ \xymatrix{
\YY'  \ar[d] \ar[r] &  \YY  \ar[d]  \\
\XX' \ar[r] &  \XX 
}\]
such that: 
\begin{enumerate}
\item  $\XX' \to \XX$ is a formal blow-up along an open, finitely generated ideal $I
\subset A_0$. 
\item $\YY'$ is the strict transform of $\YY$. 
\item $\YY' \to \XX'$ is flat. 
\end{enumerate}
Moreover, if $A_0[1/\pi] \to B_0[1/\pi]$ is faithfully flat, then $\YY' \to \XX'$
is also faithfully flat. 
\end{theorem}

We can restate the above result purely in terms of schemes (rather than formal
schemes), and in the faithfully flat case, as follows. 

\begin{proposition} 
\label{flattening2}
Let $A_0 \to B_0$ be a map in $\tfft_{\mathcal{O}_K}$ such that 
$A_0[1/\pi] \to B_0[1/\pi]$ is faithfully flat. 
Then there exists a finitely presented $A_0$-algebra $A_0'$, flat over
$\mathcal{O}_K$, such that: 
\begin{enumerate}
\item $A_0 \to A_0'$ is a universal descent morphism. 
\item  $A_0[1/\pi] \to A_0'[1/\pi]$ is faithfully flat. 
\item The map $A'_0 \to B_0' \stackrel{\mathrm{def}}{=}\mathrm{Tor}_0^{A_0}(B_0
,A_0') = \pi_0(B_0 \otimes_{A_0} A_0') $
has the property that
$A_0' \to B_0'/\left\{\pi^\infty-\mathrm{torsion}\right\}$ is faithfully flat. 
\end{enumerate}
\end{proposition} 
\begin{proof} 
This follows from \Cref{flattening}: instead of taking the formal blow-up along the
relevant (finitely generated) ideal $I \subset A_0$, we take the actual scheme-theoretic blowup
$X' \to X = \spec(A_0)$; note that $X'[1/\pi] \to X[1/\pi]$ is an isomorphism.
Since $I$ is finitely generated, $X' \to X$ is a finite type map of schemes;
since 
everything is flat over $\mathcal{O}_K$, $X' \to X$ is finitely presented
(\Cref{main:finitenessthm}). 
If we 
let $Y'$ be the strict transform of $Y = \spec(B_0)$, then the previous result states that $Y' \to X'$
defines a map of $\mathcal{O}_K$-schemes which is flat modulo $\pi$. 
We can then take $A_0'$ to be an affine cover of $X'$. 

Note that $A_0[1/\pi] \to A_0'[1/\pi]$ is
a Zariski cover, verifying (2). 
Letting $B_0'^{ \flat} =B_0'/\left\{\pi^\infty-\mathrm{torsion}\right\}$, 
the assertions of \Cref{flattening} give
that $A_0' \to B_0'^{\flat}$ becomes faithfully flat modulo any power of $\pi$. Moreover,
it is faithfully flat after inverting $\pi$ by base-change, by our assumptions. 
This implies that $A_0' \to B_0'^{\flat}$ is faithfully flat. 
Finally, to see that $A_0 \to A_0'$ is a universal descent morphism, we argue 
via \Cref{hcoverdesc}: indeed, $X' \to X$ is a proper, finitely presented and
surjective morphism and hence a $v$-cover. 
It follows that $A_0 \to A_0'$ is a finitely presented $v$-cover and hence a
universal descent morphism. 
\end{proof}

\begin{lemma}[Discreteness up to quasi-isogeny criterion] 
\label{discretenesscrit}
Let $A_0 \in \tft_{\mathcal{O}_K}$. Let $A_0'$ be a finitely presented
$A_0$-algebra such that: 
\begin{enumerate}
\item $A_0 \to A_0'$ is a  
universal descent morphism. 
\item
The map 
$A_0[1/\pi] \to A_0'[1/\pi]$ is
flat.
\end{enumerate}

Let $M \in \md(A_0)_{\geq 0}$. Suppose the (derived) base-change $M \otimes^{}_{A_0}
A_0' \in \md(A_0')$ is quasi-isogenous to a flat, discrete $A'_0$-module. 
Then each (derived) tensor power $M^{\otimes n } \in \md(A_0), n \geq 0$ is quasi-isogenous to a discrete module. 
\end{lemma} 
\begin{proof} 
Our assumptions imply that each 
$M^{\otimes n} \otimes^{\L}_{A_0} ( A_0' \otimes^{\L}_{A_0} \dots \otimes^{\L}_{A_0} A_0')$
is quasi-isogenous to a flat 
$A_0' \otimes^{\L}_{A_0} \dots \otimes^{\L}_{A_0} A_0'$-module. 
Now $A_0 \to A_0'$ 
is finitely presented as a map of rings, and hence almost finitely presented as
a map of $E_\infty$-rings (by \Cref{afpalgebra} and \Cref{main:finitenessthm}). 
Therefore, 
the iterated tensor products 
$A_0' \otimes^{\L}_{A_0} \dots \otimes^{\L}_{A_0} A_0'$ (as $E_\infty$-algebras
over $A_0$)
are almost finitely presented $E_\infty$-algebras over $A_0$ and become discrete
after inverting $\pi$, by (2). Therefore, their higher homotopy groups  are all
$\leq \pi^\infty$-isogenous to zero 
by coherence (cf.~\Cref{useofcoherence}). 
In particular, we find that 
each 
$M^{\otimes n} \otimes^{\L}_{A_0} ( A_0' \otimes^{\L}_{A_0} \dots \otimes^{\L}_{A_0} A_0')$
is quasi-isogenous to a discrete module. 
Using the canonical resolution of $M^{\otimes n}$ by its base-changes to the
iterative tensor powers $A'_0 \otimes_{A_0} \otimes \dots \otimes_{A_0} A_0'$
and 
taking some $\mathrm{Tot}^N$ (since $A_0 \to A_0'$ is universal descent), we can
conclude the
result. 
\end{proof}

\begin{proposition} 
\label{fflatgeneric}
Let $A_0 \to B_0$ be a map in $\tfft_{\mathcal{O}_K}$. 
Suppose the map $A_0[1/\pi] \to B_0[1/\pi]$ in $\affn_K$ is faithfully flat. Then: 
\begin{enumerate}
\item $A_0  \to B_0$ is a universal descent morphism. 
\item  For each $n$, the \emph{derived} tensor product $B_0
\otimes_{A_0} \dots
\otimes_{A_0} B_0 \in \md(A_0) $ has bounded $\pi$-power torsion in each
homological degree, and
is all $\pi$-torsion in positive degrees. 
\end{enumerate}
\end{proposition} 
\begin{proof} 
By \Cref{flattening2}, there exists a finitely presented, universal descent,
discrete $A_0$-algebra $A_0'$ 
such that $A_0[1/\pi] \to A_0'[1/\pi]$ is flat and 
such that
$B_0'^{\flat} \stackrel{\mathrm{def}}{=} \pi_0(B_0 \otimes_{A_0} A_0')/\left\{\pi^\infty-\mathrm{torsion}\right\}$ is
faithfully flat over $A_0'$. 
Then the (derived) tensor product $B_0 \otimes_{A_0} A_0' $ has the structure of an almost finitely
presented $E_\infty$-$B_0$-algebra which becomes discrete after inverting $\pi$. 
By coherence (\Cref{main:finitenessthm} and \Cref{useofcoherence}), $B_0 \otimes^{\L}_{A_0} A_0'$ is
quasi-isogenous to the flat, discrete
$A_0'$-module $B_0'^{\flat}$: in particular, the higher homotopy groups and the torsion in
$\pi_0$ are bounded torsion. The claim (2) now follows from  \Cref{discretenesscrit}.

Finally, we need to verify that $A_0 \to B_0$ is universal descent. 
By assumption, $A_0 \to A'_0$ is a universal descent morphism. Since the source
is $\pi$-complete, it follows that $A_0 \to \widehat{A'_0}$ is a universal
descent morphism (\Cref{picompletenesslemma}). Therefore, it suffices to show
(by two-out-of-three as in \Cref{twooutofthree}) that 
$\widehat{A'_0} \to \widehat{B_0'^{\flat}}$ is a universal descent
morphism: indeed, then $A_0 \to \widehat{B_0'^{\flat}}$ is universal descent
and hence so is $A_0 \to B_0$. 
But 
$\widehat{A'_0} \to \widehat{B_0'^{\flat}}$
is faithfully flat and finitely presented modulo $\pi$, so the
claim again follows from \Cref{picompletenesslemma}. 
\end{proof}

\begin{lemma} 
\label{picompletenesslemma}
Let $R \to S$ be a map of  $\pi$-torsion-free
$\mathcal{O}_K$-algebras.
Suppose $R$ is $\pi$-adically complete. 
Then: 
\begin{enumerate}
\item  
$R \to S$ is a universal descent morphism if and only if $R \to \widehat{S}$ is
a universal descent morphism. 
\item
 Suppose $R/\pi \to S/\pi$ is faithfully flat and
finitely presented. Then $R \to S$ is a universal descent morphism.  
\end{enumerate}
\end{lemma} 
\begin{proof} 
Let $I = \mathrm{fib}(R \to S) $ in $\md(R)$, so we have a canonical map $f: I
\to R$. 
Then $R \to S$ is a  universal descent morphism if and only if 
$f^{\otimes n}: I^{\otimes n} \to R$ is nullhomotopic (in $\md(R)$) for $n \gg 0$. 
Similarly, $R \to \widehat{S}$ is universal descent if and only if
$( \widehat{I})^{\otimes n} \to R$ is nullhomotopic for $n \gg 0$. 
These conditions are equivalent because 
$I^{\otimes n}, 
( \widehat{I})^{\otimes n}$ have the same $\pi$-completion and $R$ is
$\pi$-complete. This proves (1).

Now suppose $R/\pi \to S/\pi$ is faithfully flat and finitely presented. 
By \cite[Lemma D.3.3.7]{SAG}, 
the map $f^{\otimes n}$ vanishes after base-change to any quotient
$\mathcal{O}_K/\pi^k$ for $n \geq 2$; indeed, $I^{\otimes
n}/\pi^k$ is the $(-n)$-shift of a flat discrete $R/\pi^k$-module which is at most countably
presented, whence 
$\pi_0 \mathrm{Hom}_{R/\pi^k}( I^{\otimes n}/\pi^k, R/\pi^k) = 0$  by \emph{loc.~cit.}. 
Therefore, the map $f^{\otimes 2}: \widehat{I^{\otimes 2}} \to R$ is divisible by any power of
$\pi$. 
Consider the composable maps 
$\pi_0 \hom_R( \widehat{I^{\otimes 4}}, \widehat{I^{\otimes 4}}) 
\xrightarrow{ f^{\otimes 2} \circ \cdot}
\pi_0 \hom_R( \widehat{I^{\otimes 4}}, \widehat{I^{\otimes 2}}) 
\xrightarrow{f^{\otimes 2} \circ \cdot} 
\pi_0 \hom_R( \widehat{I^{\otimes 4}}, R) $, each of which is divisible by any
power of $\pi$; 
it follows from \Cref{compositederivedtcomplete} that the composite is zero. 
Therefore, $f^{\otimes 4}$ is nullhomotopic as desired. 
\end{proof}

\begin{theorem} 
\label{aperfrigiddesc2}
The construction $A \mapsto \aperf(A)$ defines a hypercomplete sheaf of
$\infty$-categories for the
flat topology on $(\affn_K)^{op}$. 
\end{theorem} 
\begin{proof} 
Let $A \to B$ be a faithfully flat map in $\affn_K$. 
Choose open subrings $A_0 \subset A, B_0 \subset B$ in $\tfft_{\mathcal{O}_K}$ such that $A_0$ is carried to
$B_0$. Then \v{C}ech descent of $\aperf(\cdot)$ along $A \to B$ follows from \Cref{fflatgeneric} and
\Cref{rigiddescuniv}. From this, it follows by flatness that the
subcategories $\aperf(\cdot)_{[a, b]}$ form sheaves on 
on $\affn_K$ for each $a \leq b$; these sheaves are necessarily
hypercomplete since they take values in truncated $\infty$-categories.  
Taking the limit in $a,b$, we find that $\aperf(\cdot)$ is a hypercomplete
sheaf. 
\end{proof}

\section{Flat descent}

In this section, we prove our main faithfully flat descent theorem
(\Cref{fflataperf}). 
Let $R$ be a connective $E_\infty$-ring with a finitely generated ideal 
$I \subset \pi_0(R)$. 
We need the following definition, which generalizes \Cref{rngpsite}. 
For an $E_\infty$-$R$-algebra $S$, we denote by $\md(S)_{\itors}$ the full
subcategory of $\md(S)$ spanned by the $I$-torsion modules. 

\begin{definition}[{$I$-complete flatness, after \cite[Sec.~4]{BMS2} and
\cite[Notations]{Prisms}}] 
Let $S \to S'$ be a map of connective $E_\infty$-rings under $R$. 
We will say that $S \to S'$ is \emph{$I$-completely flat} (resp.~\emph{$I$-completely faithfully flat}) if the base-change functor
$\md(S)_{\itors} \to \md(S')_{\itors}$ is $t$-exact (resp.~$t$-exact and
conservative). 
For instance, if $S \to S'$ is map of discrete rings and $I \subset R$ is
generated by a nonzerodivisor $u$, then this holds if $u$ is a nonzerodivisor in
$S, S'$ and $S/u \to S'/u$ is flat (resp.~faithfully flat). 
This defines the \emph{$I$-completely flat topology} on 
the opposite of the $\infty$-category of connective $E_\infty$-$R$-algebras. 
\end{definition}

\begin{remark} 
\label{fflatalongIhtpy}
Let $S \to S'$ be $I$-completely faithfully flat. 
Let $M \in \md(S)_{\itors}$. 
Then for each $i$ and each $r \geq 0$, the following are equivalent: 
\begin{enumerate}
\item $\pi_i(M)$ is annihilated by $I^r$.  
\item $\pi_i(S' \otimes_S M)$
is annihilated by $I^r$.  
\end{enumerate}
This follows since base-change is $t$-exact and conservative. 
\end{remark}

\begin{proposition}[{Flat hyperdescent for $\mathcal{M}(R)$}] 
\label{MRisanalyticallyflat}
The construction $S \mapsto \mathcal{M}(S)$ defines a hypercomplete sheaf for the
$I$-completely flat topology. 
Similarly for $S \mapsto \mathcal{M}(S)_{\geq 0}$. 
\end{proposition} 
\begin{proof} 
This follows from \Cref{descleftcompl}. Indeed, 
if $S \to S'$ is $I$-completely faithfully flat, then 
the functor 
$\mathcal{M}(S) \to \mathcal{M}(S')$ of 
stable $\infty$-categories equipped with 
right-bounded, left-complete $t$-structures (specifically, the $I$-torsion
$t$-structures) is conservative and $t$-exact.  
Therefore, we find that for each $m \leq n$, $S \mapsto \mathcal{M}(S)_{[m, n]}$
is a sheaf for the $I$-completely flat topology (necessarily hypercomplete
since these are truncated $\infty$-categories); taking the inverse limit over
$n$ and then the direct limit over $m$,  we find that $S \mapsto \mathcal{M}(S)$
is a hypercomplete sheaf. 
Similarly for $S \mapsto \mathcal{M}(S)_{\geq 0}$. 
\end{proof}

\begin{proposition} 
Suppose $S \to S'$ is an $I$-completely faithfully flat map of connective
$E_\infty$-$R$-algebras. Let $M \in \mdc{S}$. 
Then the following are equivalent: 
\begin{enumerate}
\item $\pi_i(M)$ is $\leq I^\infty$-isogenous to zero for $i < 0$ (i.e., $M$
defines a connective object of $\mathcal{M}_0(S)$).  
\item 
$\pi_i( \widehat{M \otimes_S S'})$ is $\leq I^\infty$-isogenous to zero for $i
< 0$ (i.e., $\widehat{M \otimes_S S'}$ defines a connective object of
$\mathcal{M}_0(S')$). 
\end{enumerate}
\label{descentof}
\end{proposition} 
\begin{proof} 
Fix a tower $\left\{S_n\right\}$ as in \Cref{towercompletion}.
By \Cref{connectivityIbound}, the first condition is equivalent to the assertion
that for each $i < 0$, the modules $\pi_i ( M \otimes_S S_n)$ are $\leq
I^{r_1}$-isogenous to zero for some $r_1$ and all $n$.  
Similarly, the second condition is equivalent to the assertion that for each 
$i  < 0$, 
the modules $\pi_i ( (M \otimes_S S_n) \otimes_S S')$ are $\leq
I^{r_2}$-isogenous to zero for some $r_2$ and all $n$. 
But these two conditions are then equivalent as in \Cref{fflatalongIhtpy}. \end{proof}

\begin{proposition} 
\label{descfg}
Let $S \to S'$ be an $I$-completely faithfully flat map of connective
$E_\infty$-$R$-algebras. 
Let $M \in \md(S)_{\itors}$ be connective. 
Then for any $S$, the following are equivalent: 
\begin{enumerate}
\item $M$ is $\leq I^r$-perfect to order zero as an $S$-module. 
\item $S' \otimes_S M$  is $\leq
I^r$-perfect to order zero as an $S'$-module. 
\end{enumerate}
\end{proposition} 
\begin{proof} 
Clearly (1) implies (2). 
Suppose (2). 
Fix a tower $\left\{S_n\right\}$ as in \Cref{towercompletion}.
Then 
there exists a map $f_{1} \cl S'^m \to S' \otimes_S M$ whose cofiber
has $\pi_0$ annihilated by 
$I^r$. Up to increasing $m$, we can 
assume $f_1$ is the base-change of a map $f_0 \cl S^m \to M$. 
This factors over a map $f_0  \cl S_n^m \to M$ for $n \gg 0$, since $M$ is torsion. 
It now follows by descent (\Cref{fflatalongIhtpy}) that
$\pi_0(\mathrm{cofib}(f_0))$ is annihilated by
$\leq I^r$. Therefore, $M$ is $\leq I^r$-perfect to order zero. 
\end{proof} 

\begin{proposition} 
Suppose $S \to S'$ is an $I$-completely faithfully flat map of
$I$-complete connective
$E_\infty$-$R$-algebras.
Let $M \in \mdc{S}$. 
Then the following are equivalent: 
\begin{enumerate}
\item $M$ is $\leq I^\infty$-perfect to order $n$.   
\item $\widehat{M \otimes_S S'}$  is $\leq I^\infty$-perfect to order $n$. 
\end{enumerate}
\label{descfgI}
\end{proposition} 
\begin{proof} 
Without loss of generality, we can assume that $M$ is actually connective as an
$R$-module. 
It suffices to show that (2) implies (1). 
We use induction on $n$. 
Suppose $n = 0$. 
Fix a tower $\left\{S_m\right\}$ as in \Cref{towercompletion}.
By \Cref{cor:fgaway}, the first condition is equivalent to the assertion
that there exists some $r_1$ such that $M \otimes_S S_m$ is $\leq
I^{r_1}$-perfect to order zero for all $m$. Similarly, the second condition 
is equivalent to the assertion
that there exists some $r_2$ such that $M \otimes_S S_m \otimes_S S'$ is
$\leq I^{r_2}$-perfect to order zero for all $m$. The two conditions are thus equivalent 
thanks to \Cref{descfg}. 

Suppose (2) 
with $n > 0$. From what we have already shown, $\pi_0(M)$ is $\leq
I^\infty$-finitely generated, and up to modifying $M$ up to $\leq
I^\infty$-isogeny we may assume that $\pi_0(M)$ is finitely generated. 
Replacing $M$ with the fiber of a map $S^r \to M$ inducing a surjection on
$\pi_0$ and using 
\Cref{prop:shiftIperfect} (as well as the inductive step), we can now conclude
(1) for $M$. 
\end{proof} 

\begin{proposition} 
Suppose $S \to S'$ is an $I$-completely faithfully flat map of
$I$-complete connective
$E_\infty$-$R$-algebras.
Let $M \in \mathcal{M}(S)$. Then the following are equivalent: 
\begin{enumerate}
\item $M$ is weakly perfect to order $n$.  
\item The base-change $S' \otimes_S M \in \mathcal{M}(S')$ is weakly perfect to order $n$. 
\end{enumerate}

\label{localofweakperfect}
\end{proposition} 
\begin{proof} 
Both of these conditions only depend on $\tau_{\leq n} M$, so we can assume $M
\in \mathcal{M}(S)_{\leq n} = \mathcal{M}_0(S)_{\leq n}$ and represent $M$ by an
object $M'$ of $\mdc{S}$. 
Then 
condition (1) is equivalent to the assertion that $M'$ is $\leq
I^\infty$-perfect to order $n$ while condition (2) is equivalent to the
assertion that $\widehat{M \otimes_S S'}$ is $\leq I^\infty$-perfect  to order
$n$. 
The result follows from \Cref{descfgI}. 
\end{proof} 

Now we can prove 
the main descent theorem of this article, which generalizes the results of
Drinfeld \cite[Th.~3.11]{Dri06} and \cite[Prop.~3.5.4]{Dri18}. 
Heuristically, 
\Cref{fflataperf} suggests that almost perfect
complexes (defined purely algebraically) behave well in analytic geometry quite generally. Closely
related is Kedlaya--Liu \cite[Sec.~2]{KL16} 
which constructs a category of pseudocoherent sheaves on adic spaces (corresponding to almost
perfect complexes which are discrete).
See also Hennion--Porta--Vezzosi \cite{HPV} for the case of algebras finite type over a
field with the \'etale topology, and generalizations to stacks. 
We refer to \cite{AnalyticGeometry} for a new approach to defining analogs of
``big'' categories of quasi-coherent sheaves in analytic geometry, from which
we expect it should be possible to recover our results.\footnote{Compare the
recent work of Andreychev, \cite{And21}.} 

\begin{theorem}[Faithfully flat descent for $\aperf$] 
\label{fflataperf}
The construction $S \mapsto \aperf( \spec(\hat{S}_I) \setminus V(I))$ defines a
hypercomplete sheaf for the $I$-completely 
flat topology. 
Similarly for the subcategories $\aperf_{\geq 0}$ of connective almost perfect
modules,  $\perf$ of perfect modules, and $\perf_{[a,b]}$ of perfect modules
with $\mathrm{Tor}$-amplitude in $[a,b]$
for any $a \leq b$. 
\end{theorem} 
\begin{proof} 
First, we have seen that $S \mapsto \mathcal{M}(S)$ is a hypercomplete sheaf of 
$\infty$-categories (\Cref{MRisanalyticallyflat}). 
Second, there is an embedding $\aperf( \spec(S) \setminus V(I)) \subset
\mathcal{M}(S)$ (whose inverse is given by the functor $j^*$), by
\Cref{jgivesequiv}. 
Third, 
the condition of belonging to 
$\aperf_{}$ is local 
(\Cref{localofweakperfect}). 
Combining these three assertions, we find  that 
$S \mapsto \aperf( \spec(\hat{S}_I) \setminus V(I))_{}$ is a
hypercomplete sheaf. 

We can carry out a similar 
argument for $\aperf(\cdot)_{\geq 0}$, since we also know that $S \mapsto
\mathcal{M}(S)_{\geq 0}$ is a hypercomplete sheaf as well and $\aperf( \spec( \widehat{S}_I)
\setminus V(I))_{\geq 0} \subset \mathcal{M}(S)_{\geq 0}$ (as in
\emph{loc.~cit.}). The result for $\aperf$ implies the result 
for $\perf$ by taking the subcategories of dualizable objects. The result for
$\perf_{\geq 0}$ now follows by taking the intersection of $\aperf_{\geq 0}$ and
$\perf$. Finally, $\perf_{[a, b]}$ is given by those objects of $\perf$ which
belong to $\perf_{\geq a}$ and whose dual belongs to $\perf_{\geq -b}$, so this
is also a local condition. 
\end{proof} 

\bibliographystyle{plain}
\bibliography{D}

\end{document}